\documentclass[portuges,12pt,letter]{article}
\usepackage[centertags]{amsmath}
\usepackage{amsfonts}
\usepackage{newlfont}
\usepackage{amscd}
\usepackage{graphics}
\usepackage{epsfig}
\usepackage{indentfirst}
\usepackage{amsxtra}
\usepackage[latin1]{inputenc}
\usepackage{amssymb, amsmath}
\usepackage{amsthm}
\usepackage[mathscr]{eucal}

\newtheorem{thm}{Theorem}[section]

\newtheorem{remark}[thm]{Remark}
\author{Fabio Silva Botelho \\ Department of Mathematics \\ Federal University of  Santa Catarina, UFSC \\
Florian\'{o}polis, SC - Brazil}
\date{}
\title{\bf On the generalized method of lines applied to  the time-independent incompressible Navier-Stokes system } %and a numerical method for related models}

\begin{document}
\maketitle

\abstract{
In the first part of this article, we obtain a linear system whose the solution solves the time-independent incompressible Navier-Stokes
system for the special case in which the external forces vector is a gradient. In a second step we develop approximate solutions, also for the time independent incompressible Navier-Stokes system, through
the generalized method of lines. We recall that for such a method, the domain of the partial differential equation
in question is discretized in lines  and the concerning solution is written on these lines as functions of
the boundary conditions and  boundary shape.
Finally, we emphasize these last main results are established through  applications of the Banach fixed point theorem.
%Insert your abstract here. Include keywords, PACS and mathematical
%subject classification numbers as needed.
%\keywords{Primal dual formulation\and  \and Duality \ Non-linear plate model} %First keyword \and Second keyword \and More}
% \PACS{PACS code1 \and PACS code2 \and more}
%\subclass{ 49J10 \and  49N15}%\and MSC code2 \and more}}
}
 %In a second step we develop a new matrix version of the generalized method of lines, applicable to a large class of related models.}
\section{Introduction }
This article develops solutions for the time-independent incompressible Navier-Stokes system through the generalized method of lines.
At this point we describe the system in question.

Consider $\Omega \subset \mathbb{R}^2$ an open, bounded and connected
set, whose the regular (Lipschitzian) internal boundary is denoted by $\Gamma_0 $ and the
regular external  one is denoted by $\Gamma_1 $. For a two-dimensional motion of a fluid on $\Omega$, we denote by $u:\Omega
\rightarrow \mathbb{R}$ the  velocity  field in the direction $x$ of
the Cartesian system $(x,y)$, by $v:\Omega \rightarrow \mathbb{R}$, the
velocity field in the direction $y$ and  by $p:\Omega \rightarrow
\mathbb{R}$, the pressure one. We define $P=p/ \rho$, where $\rho$
is the constant fluid density. Finally, $\nu$ denotes the viscosity coefficient
and $g$ denotes the gravity field. Under such notation and statements, the time-independent incompressible   Navier-Stokes   system of partial
differential equations is
expressed by,
\begin{gather}\label{910} \left\{
\begin{array}{ll}
 \nu \nabla^2 u-u \partial_x u-v \partial_y u- \partial_x P+g_x=0, & \text{ in }  \Omega, \\[8pt]
 \nu \nabla^2 v-u \partial_x v-v \partial_y v- \partial_y P+g_y=0, & \text{ in } \Omega,\\[8pt]
\partial_x u+ \partial_y v=0, & \text{ in } \Omega,
\end{array} \right.\end{gather}
\begin{gather}\label{911} \left\{
\begin{array}{ll}
 u=v=0, & \text{ on }  \Gamma_0,\\[8pt]
u=u_\infty, \; v=0, \; P=P_\infty, & \text{ on } \Gamma_1
\end{array} \right.\end{gather}

In principle we look for solutions $(u,v,P) \in W^{2,2}(\Omega)\times W^{2,2}(\Omega) \times W^{1,2}(\Omega)$ despite the
fact that less regular solutions are also possible specially concerning the weak formulation. Details about such Sobolev spaces may be found in \cite{1}.
General results on finite differences and existence theory for similar systems may be found in \cite{103} and \cite{890}, respectively.
\section{On the solution of  the  time-independent incompressible Navier-Stokes system through an associated linear one}

 Through the next result we obtain a linear system whose the solution also solves the time-independent incompressible Navier-Stokes system
 for the special case in which the external forces vector is a gradient.

 Similar results for the time-independent incompressible Euler and Navier-Stokes equations have been presented in \cite{120} and \cite{901,910}, respectively.

 Indeed in the works  \cite{901,910}, we have indicated a solution of the Navier-Stokes  given by
 $\mathbf{u}=(u,v)$ defined by
 \begin{gather}\label{918} \left\{
\begin{array}{l}
u=\partial_x
w_0+\partial_x w_1,
 \\ \\
v=\partial_y  w_0-\partial_y w_1,
\end{array} \right.\end{gather}
 where $w_0,w_1$ are solutions of the system
\begin{gather}\label{914} \left\{
\begin{array}{ll}
\partial_{xy} w_1=0 & \text{ in } \Omega,
\\
\nabla^2 w_0+\partial_{xx}w_1-\partial_{yy}w_1=0,& \text{ in } \Omega,
 \\
 u=u_0,& \text{ on } \Gamma,
 \\
 v=v_0,& \text{ on } \Gamma.
\end{array} \right.\end{gather}

Thus, in such a sense, the next result complements this previous one, by introducing a new function $w_2$ in the solution expressions, which
makes the concerning boundary conditions perfectly possible to be satisfied.

\begin{thm}\label{990} For $h=(\partial_x f,\partial_y f) \in C^1(\Omega;\mathbb{R}^2)$, consider the Navier-Stokes system similar as above indicated, that is,
\begin{gather}\label{912} \left\{
\begin{array}{ll}
 \nu \nabla^2u-u \partial_x u-v \partial_y u- \partial_x P+\partial_x f=0, & \text{ in }  \Omega,
 \\ \\
 \nu \nabla^2 v-u \partial_x v-v \partial_y v- \partial_y P+\partial_y f=0, & \text{ in } \Omega,
\\ \\
\partial_x u+ \partial_y v=0, & \text{ in } \Omega,
\end{array} \right.\end{gather}
with the boundary conditions

\begin{gather}\label{913} \left\{
\begin{array}{ll}
u=u_0,& \text{ on } \Gamma\equiv\Gamma_0 \cup \Gamma_1\\
v=v_0,& \text{ on } \Gamma,
 \\
 P=P_0, & \text{ on } \Gamma_1.
\end{array} \right.\end{gather}

A solution for such a Navier-Stokes system is given by
  $\mathbf{u}=(u,v)$ defined by
 \begin{gather}\label{918} \left\{
\begin{array}{l}
u=\partial_x
w_0+\partial_x w_1+\partial_y w_2,
 \\ \\
v=\partial_y  w_0-\partial_y w_1-\partial_x w_2,
\end{array} \right.\end{gather}
 where $w_0,w_1,w_2$ are solutions of the system
\begin{gather}\label{914} \left\{
\begin{array}{ll}
\nabla^2 w_2+2\partial_{xy} w_1=0 & \text{ in } \Omega,
\\
\nabla^2 w_0+\partial_{xx}w_1-\partial_{yy}w_1=0,& \text{ in } \Omega,
 \\
 u=u_0,& \text{ on } \Gamma,
 \\
 v=v_0,& \text{ on } \Gamma,
\end{array} \right.\end{gather}
and $P$ is a solution of the system indicated in  the first two lines of (\ref{912}) with boundary conditions indicated in the third line of (\ref{913}).
\end{thm}
\begin{proof} For $w_0, w_1, w_2$ such that $\nabla^2 w_2+2\partial_{xy} w_1=0 \text{ in } \Omega$, $u$ and $v$ as indicated above and
defining $$h_1=u \partial_x u +v \partial_y u,$$ $$h_2=u \partial_x v +v \partial_y v$$
and $\varphi\equiv \nabla^2 w_2+2\partial_{xy} w_1=0,$ we have (you may check it using the softwares MATHEMATICA or MAPLE)
\begin{eqnarray}
&&\frac{\partial h_1}{\partial y}-\frac{\partial h_2}{\partial x}\nonumber \\ &=&
(-\partial_{yy}w_1+\partial_{yy} w_0-\partial_{xy} w_2) \varphi \nonumber \\ &&+(-\partial_y w_1+\partial_y w_0-\partial_x w_2) \partial_y \varphi
\nonumber \\ && +(\partial_{xy} w_2+\partial_{xx} w_1+\partial_{xx}w_0) \varphi
\nonumber \\ && +(\partial_y w_2+\partial_x w_1+\partial_x w_0) \partial_x \varphi \nonumber \\ &=& 0, \text{ in } \Omega.
\end{eqnarray}

Moreover, since $\varphi=\nabla^2 w_2+2\partial_{xy} w_1=0 \text{ in } \Omega$, we get
\begin{eqnarray}&&\nu(\partial_y \nabla^2u-\partial_x \nabla^2 v)+\partial_y (\partial_x f)-\partial_x(\partial_y f)
\nonumber \\ &=&\nu(2\partial_{xy}(\nabla^2w_1)+\nabla^4 w_2)+ \partial_y (\partial_x f)-\partial_x(\partial_y f)
\nonumber \\ &=& \nu \nabla^2 \varphi+\partial_{yx}f -\partial_{yx} f \nonumber \\ &=&0,\text{ in }  \Omega. \end{eqnarray}

Summarizing, we have obtained
\begin{eqnarray} && \partial_y\left(\nu \nabla^2u-u \partial_x u-v \partial_y u+\partial_x f\right)
\nonumber \\ &=& \partial_x\left(\nu \nabla^2 v-u \partial_x v-v \partial_y v+\partial_y f\right).
\end{eqnarray}

Also, the equation
$$\nabla^2 w_0+\partial_{xx}w_1-\partial_{yy}w_1=0,\text{ in }  \Omega$$
stands for $$\partial_x u+\partial_y v=0, \text{ in } \Omega.$$
From these last results we may obtain $P$ which satisfies the concerning boundary condition such that
\begin{gather}\label{912a} \left\{
\begin{array}{ll}
 \nu \nabla^2u-u \partial_x u-v \partial_y u- \partial_x P+\partial_x f=0, & \text{ in }  \Omega,
 \\ \\
 \nu \nabla^2v-u \partial_x v-v \partial_y v- \partial_y P+\partial_y f=0, & \text{ in } \Omega,
\\ \\
\partial_x u+ \partial_y v=0, & \text{ in } \Omega,
\end{array} \right.\end{gather}
This completes the proof.
\end{proof}
\section{The generalized method of lines for the Navier-Stokes system}

In this section we develop the solution for the Navier-Stokes system through the generalized method of lines, which was originally introduced in \cite{901}, with further developments in \cite{909,120}.
%\subsection{A m general case for specific boundary conditions}
We consider  the boundary conditions $$u=u_0(x), \;\; v=v_0(x),\; P=P_0(x)\text{ on } \partial \Omega_0,$$
$$u=0,\;v=0, \;P=P_f(x) \text{ on } \partial \Omega_1,$$ where $$\Omega=\{(r,\theta): \;|\;  r(\theta) \leq r \leq 2r(\theta)\}$$
and where $r(\theta)$ is a positive, smooth  and periodic function with period $2\pi$, $$\partial \Omega=\partial \Omega_0 \cup \partial \Omega_1,$$
$$\partial \Omega_0=\{(r(\theta),\theta) \in \mathbb{R}^2\;:\; 0 \leq \theta \leq 2\pi\}$$ and
 $$\partial \Omega_1=\{(2r(\theta),\theta) \in \mathbb{R}^2\;:\; 0 \leq \theta \leq 2\pi\}.$$

For $\nu=1$, neglecting the gravity effects, the corresponding Navier-Stokes homogeneous system, in function of the variables $(t,\theta)$ where $t=r/r(\theta),$ is given by
\begin{equation}\label{us1234}L(u)-ud_1(u)-vd_2(u)-d_1(P)=0,\end{equation}
\begin{equation}\label{us1235}L(v)-ud_1(v)-vd_2(v)-d_2(P)=0,\end{equation}
\begin{equation}\label{us1236}d_1(u)+d_2(v)=0,\end{equation}
where generically  $$L(u)=\nabla^2u,$$ $$d_1(u)=\partial_x u,$$  and $$d_ 2(u)=\partial_y u$$   will be specified in the next lines, in function of $(t,\theta)$.

Firstly, $L$ is such that
\begin{eqnarray}
L(u)\left(\frac{r(\theta)^2}{f_0(\theta)}\right)=\frac{\partial^2 u}{\partial t^2}&+&\frac{1}{t}f_2(\theta)\frac{\partial u}{\partial t}\nonumber \\
&+&\frac{1}{t}f_3(\theta)\frac{\partial^2 u}{\partial \theta \partial t}+\frac{f_4(\theta)}{t^2} \frac{\partial^2 u}{\partial \theta^2},
\end{eqnarray}
in $\Omega$. Here $f_0(\theta), f_2(\theta), \;f_3(\theta)$ and $f_4(\theta)$ are known functions.

More specifically, denoting $$f_1(\theta)=\frac{-r'(\theta)}{r(\theta)},$$
we have:
$$f_0(\theta)=1+f_1(\theta)^2,$$
$$f_2(\theta)=1+\frac{f_1'(\theta)}{1+f_1(\theta)^2},$$ $$f_3(\theta)=\frac{2f_1(\theta)}{1+f_1(\theta)^2},$$ and
$$f_4(\theta)=\frac{1}{1+f_1(\theta)^2}.$$
Also $d_1$ and $d_2$ are expressed by $$d_1u = \hat{f}_5(\theta) \frac{\partial u}{\partial t} + (\hat{f}_6(\theta)/t)\frac{\partial u}{\partial \theta},$$
$$d_2u = \hat{f}_7(\theta) \frac{\partial u}{\partial t} + (\hat{f}_{8}(\theta)/t)\frac{\partial u}{\partial \theta}.$$
Where $$\hat{f}_5(\theta)=\cos(\theta)/r(\theta)+\sin(\theta)r'(\theta)/r^2(\theta),$$ $$\hat{f}_6(\theta)=-\sin(\theta)/r(\theta),$$
$$\hat{f}_7(\theta)=\sin(\theta)/r(\theta)-\cos(\theta)r'(\theta)/r^2(\theta),$$ $$\hat{f}_8(\theta)=\cos(\theta)/r(\theta).$$

We also define $$h_3(\theta)=\frac{f_0(\theta)}{r(\theta)^2},$$ $$f_5(\theta)=\left(\frac{r(\theta)^2}{f_0(\theta)}\right)\hat{f}_5(\theta),$$
$$f_6(\theta)=\left(\frac{r(\theta)^2}{f_0(\theta)}\right)\hat{f}_6(\theta),$$
$$f_7(\theta)=\left(\frac{r(\theta)^2}{f_0(\theta)}\right)\hat{f}_7(\theta),$$
and
$$f_8(\theta)=\left(\frac{r(\theta)^2}{f_0(\theta)}\right)\hat{f}_8(\theta),$$
Observe that $t \in [1,2]$ in $\Omega$.

From equations (\ref{us1234}) and (\ref{us1235})  we may write
\begin{eqnarray}\label{us1237}
&&d_1(L(u)-ud_1(u)-vd_2(u)-d_1(P)) \nonumber \\ &&+d_2(L(v)-ud_1(v)-vd_2(v)-d_2(P))=0,
\end{eqnarray}
From (\ref{us1236}) we have $$d_1[L(u)]+d_2[L(v)]=L(d_1 (u)+d_2(v))=0,$$ and considering that
$$d_1(d_1(P))+d_2(d_2(P))=L(P),$$ from (\ref{us1237}) we have
$$L(P)+d_1(u)^2+d_2(v)^2+2d_2(u)d_1(v)=0, \text{ in }\Omega.$$

Hence, in fact we solve the approximate system (this indeed is not exactly the Navier-Stokes one):
\begin{equation}\label{eq1}L(u)-ud_1(u)-vd_2(u)-d_1(P)=0,\end{equation}
\begin{equation}\label{eq2}L(v)-ud_1(v)-vd_2(v)-d_2(P)=0,\end{equation}
\begin{equation}\label{eq3}L(P)+d_1(u)^2+d_2(v)^2+2d_2(u)d_1(v)=0, \text{ in }\Omega.\end{equation}
\begin{remark} Who taught me how to obtain this last approximate system was Professor Alvaro de Bortoli of Federal University of Rio Grande do Sul,
UFRGS, Porto Alegre, RS- Brazil.
\end{remark}

At this point, discretizing only in $t$ (in $N$ lines), defining $d=1/N$ and $t_n=1+nd$, $\forall n \in \{1,\ldots,N-1\}$,
 we represent such a concerning system in partial finite differences.

Denoting $$\hat{d}_1(u_n,u_{n-1})=f_5(x)\frac{(u_n-u_{n-1})}{d}+\frac{f_6(x)}{t_n} \frac{\partial u_n}{\partial x},$$
and
$$\hat{d}_2(u_n,u_{n-1})=f_7(x)\frac{(u_n-u_{n-1})}{d}+\frac{f_8(x)}{t_n} \frac{\partial u_n}{\partial x},$$
where $x$ stands for $\theta$, in partial finite differences, equation (\ref{eq1}) stands for
\begin{eqnarray}
&&\frac{u_{n+1}-2u_n+u_{n-1}}{d^2}+\frac{f_2(x)}{t_n}\frac{u_n-u_{n-1}}{d}
\nonumber \\ &&+\frac{f_3(x)}{t_n}\frac{\partial}{\partial x}\left(\frac{u_{n}-u_{n-1}}{d}\right)+\frac{f_4(x)}{t_n^2}\frac{\partial^2 u_n}{\partial x^2}
\nonumber \\ &&-u_n \hat{d}_1(u_n,u_{n-1})-v_n \hat{d}_2(u_n,u_{n-1})-\hat{d}_1(P_n,P_{n-1})=0.
\end{eqnarray}

Hence, denoting $\mathbf{u}=(u,v,P)$, we have $$u_n=(T_1)_n(\mathbf{u}_{n+1},\mathbf{u}_n,\mathbf{u}_{n-1}),$$
where
\begin{eqnarray}
&&(T_1)_n(\mathbf{u}_{n+1},\mathbf{u}_n,\mathbf{u}_{n-1})\nonumber \\ &=&\left(u_{n+1}+u_n+u_{n-1}+\frac{f_2(x)}{t_n}(u_{n}-u_{n-1})d\right.
\nonumber \\ &&+\frac{f_3(x)}{t_n}\frac{\partial}{\partial x}(u_{n}-u_{n-1})d+\frac{f_4(x)}{t_n^2}\frac{\partial^2 u_n}{\partial^2 x}d^2
\nonumber \\ &&\left.-u_n \hat{d}_1(u_n,u_{n-1})d^2-v_n \hat{d}_2(u_n,u_{n-1})d^2-\hat{d}_1(P_n,P_{n-1})d^2\right)/3.0.\nonumber
\end{eqnarray}

Similarly, equation (\ref{eq2}) stands for

%Thus, in partial finite differences, equation (\ref{eq1}) stands for
\begin{eqnarray}
&&\frac{v_{n+1}-2v_n+v_{n-1}}{d^2}+\frac{f_2(x)}{t_n}\frac{v_{n}-v_{n-1}}{d}
\nonumber \\ &&+\frac{f_3(x)}{t_n}\frac{\partial}{\partial x}\left(\frac{v_{n}-v_{n-1}}{d}\right)+\frac{f_4(x)}{t_n^2}\frac{\partial^2 v_n}{\partial x^2}
\nonumber \\ &&-u_n \hat{d}_1(v_n,v_{n-1})-v_n \hat{d}_2(v_n,v_{n-1})-\hat{d}_2(P_n,P_{n-1})=0.
\end{eqnarray}

Hence $$v_n=(T_2)_n(\mathbf{u}_{n+1},\mathbf{u}_n,\mathbf{u}_{n-1}),$$
where
\begin{eqnarray}
&&(T_2)_n(\mathbf{u}_{n+1},\mathbf{u}_n,\mathbf{u}_{n-1})\nonumber \\ &=&\left(v_{n+1}+v_n+v_{n-1}+\frac{f_2(x)}{t_n}(v_{n}-v_{n-1})d\right.
\nonumber \\ &&+\frac{f_3(x)}{t_n}\frac{\partial}{\partial x}(v_{n}-v_{n-1})d+\frac{f_4(x)}{t_n^2}\frac{\partial^2 v_n}{\partial x^2}d^2
\nonumber \\ &&\left.-u_n \hat{d}_1(v_n,v_{n-1})d^2-v_n \hat{d}_2(v_n,v_{n-1})d^2-\hat{d}_2(P_n,P_{n-1})d^2\right)/3.0. \nonumber
\end{eqnarray}

Finally, (\ref{eq3}) stands for
\begin{eqnarray}
&&\frac{P_{n+1}-2P_n+P_{n-1}}{d^2}+\frac{f_2(x)}{t_n}\frac{P_{n}-P_{n-1}}{d}
\nonumber \\ &&+\frac{f_3(x)}{t_n}\frac{\partial}{\partial x}\left(\frac{P_{n}-P_{n-1}}{d}\right)+\frac{f_4(x)}{t_n^2}\frac{\partial^2 P_n}{\partial x^2}
\nonumber \\ &&+ h_3(x) \left(\hat{d}_1(u_n,u_{n-1})^2+\hat{d}_2(v_n,v_{n-1})^2+2 \hat{d}_2(u_n,u_{n-1})\hat{d}_1(v_,v_{n-1})\right)=0.
\end{eqnarray}

Hence $$P_n=(T_3)_n(\mathbf{u}_{n+1},\mathbf{u}_n,\mathbf{u}_{n-1}),$$
where
\begin{eqnarray}
&&(T_3)_n(\mathbf{u}_{n+1},\mathbf{u}_n,\mathbf{u}_{n-1})\nonumber \\ &=&\left(P_{n+1}+P_n+P_{n-1}+\frac{f_2(x)}{t_n}(P_{n}-P_{n-1})d\right.
\nonumber \\ &&+\frac{f_3(x)}{t_n}\frac{\partial}{\partial x}(P_{n}-P_{n-1})d+\frac{f_4(x)}{t_n^2}\frac{\partial^2 P_n}{\partial x^2}d^2
\nonumber \\ &&\left. +h_3(x)\left(\hat{d}_1(u_n,u_{n-1})^2d^2+\hat{d}_2(v_n,v_{n-1})^2d^2+2 \hat{d}_2(u_n,u_{n-1})\hat{d}_1(v_,v_{n-1})d^2\right)\right)/3.0. \nonumber
\end{eqnarray}

Summarizing, we may write
$$\mathbf{u}_n= \hat{T}_n(\mathbf{u}_{n+1},\mathbf{u}_n,\mathbf{u}_{n-1}),$$
where \begin{eqnarray}&&\hat{T}_n(\mathbf{u}_{n+1},\mathbf{u}_n,\mathbf{u}_{n-1})\nonumber \\ &=&((T_1)_n(\mathbf{u}_{n+1},\mathbf{u}_n,\mathbf{u}_{n-1}),
(T_2)_n(\mathbf{u}_{n+1},\mathbf{u}_n,\mathbf{u}_{n-1}),(T_3)_n(\mathbf{u}_{n+1},\mathbf{u}_n,\mathbf{u}_{n-1})),
 \end{eqnarray}
 $\forall n \in \{1,\ldots,N-1\}.$

Therefore, for $n=1$ we obtain

$$\mathbf{u}_1= \hat{T}_1(\mathbf{u}_2,\mathbf{u}_1,\mathbf{u}_0).$$

We solve such an equation through the Banach fixed point theorem.

\begin{enumerate}
\item First set $$(\mathbf{u}_1)^1=\mathbf{u}_2.$$
\item In a second step, define $\{\mathbf{u}_1^k\}$ such that
$$\mathbf{u}_1^{k+1}=\hat{T}_1(\mathbf{u}_2,\mathbf{u}_1^k,\mathbf{u}_0),\; \forall k \in \mathbb{N}.$$
\item Finally obtain
$$\mathbf{u}_1=\lim_{k \rightarrow \infty} \mathbf{u}_1^k \equiv F_1(\mathbf{u}_2,\mathbf{u}_0).$$
\end{enumerate}

Now, reasoning inductively, having  $$\mathbf{u}_{n-1}=F_{n-1}(\mathbf{u}_n,\mathbf{u}_0),$$ we obtain $\mathbf{u}_{n}$ as indicated
in the next lines.
\begin{enumerate}
\item First set $$(\mathbf{u}_n)^1=\mathbf{u}_{n+1}.$$
\item In a second step, define $\{\mathbf{u}_n^k\}$ such that
$$\mathbf{u}_n^{k+1}=\hat{T}_n(\mathbf{u}_{n+1},\mathbf{u}_n^k,\mathbf{u}_0),\; \forall k \in \mathbb{N}.$$
\item Finally obtain
$$\mathbf{u}_n=\lim_{k \rightarrow \infty} \mathbf{u}_n^k\equiv F_n(\mathbf{u}_{n+1},\mathbf{u}_0).$$
\end{enumerate}

Thus, reasoning inductively we have obtained $$\mathbf{u}_n=F_n(\mathbf{u}_{n+1},\mathbf{u}_0),\; \forall n \in \{1,\ldots,N-1\}.$$

In particular, for $n=N-1,$ we have $\mathbf{u}_N=\mathbf{u}_f=(u_f,v_f,P_f).$

Therefore $$\mathbf{u}_{N-1}=F_{N-1}(\mathbf{u}_N,\mathbf{u}_0)\equiv H_{N-1}(\mathbf{u}_f,\mathbf{u}_0).$$

From this we obtain
$$\mathbf{u}_{N-2}=F_{N-2}(\mathbf{u}_{N-1},\mathbf{u}_0)\equiv H_{N-2}(\mathbf{u}_f,\mathbf{u}_0),$$
and so on, up to finding
$$\mathbf{u}_1=F_{1}(\mathbf{u}_{2},\mathbf{u}_0)\equiv H_{1}(\mathbf{u}_f,\mathbf{u}_0).$$

The problem is then solved.

With such  results in mind, with a software similar to those presented in \cite{120A}, truncating the concerning series solutions for terms
  of order up to $d^2$ (in $d$), for the field of velocity $u$ we have obtained the following expressions for the lines (here $x$ stands for $\theta$):

\begin{eqnarray}
Line\;1&& \nonumber \\
u_1(x)&=& -0.045 f_5(x) P_f(x)+0.045 f_5(x) P_0(x)
\nonumber \\ &&+0.899 u_0(x)-0.034 f_2(x)u_0(x) +0.029 f_5(x) u_0(x)^2
\nonumber \\ && +0.029 f_7(x) u_0(x)v_0(x)-0.011 f_6(x) P_f'\nonumber \\ &&
 -0.022 f_6(x) P_0'(x)-0.034 f_3(x) u_0'(x) \nonumber \\ && -0.016 f_6(x)u_0(x)u_0'(x)-0.016 f_8(x) v_0(x) u_0'(x)
 \nonumber \\ && +0.018 f_4(x)u_0''(x) \nonumber \end{eqnarray}

 \begin{eqnarray}
Line\;2&& \nonumber \\
u_2(x)&=& -0.081 f_5(x) P_f(x)+0.081 f_5(x) P_0(x)
\nonumber \\ &&+0.799 u_0(x)-0.034 f_2(x)u_0(x) +0.059 f_5(x) u_0(x)^2
\nonumber \\ && +0.048 f_7(x) u_0(x)v_0(x)-0.022 f_6(x) P_f'\nonumber \\ &&
 -0.036 f_6(x) P_0'(x)-0.059 f_3(x) u_0'(x) \nonumber \\ &&-0.025 f_6(x)u_0(x)u_0'(x)-0.025 f_8(x) v_0(x) u_0'(x)
 \nonumber \\ && +0.028 f_4(x)u_0''(x) \nonumber \end{eqnarray}

\begin{eqnarray}
Line\;3&& \nonumber \\
u_3(x)&=& -0.106 f_5(x) P_f(x)+0.106 f_5(x) P_0(x)
\nonumber \\ &&+0.698 u_0(x)-0.075 f_2(x)u_0(x) +0.060 f_5(x) u_0(x)^2
\nonumber \\ && +0.060 f_7(x) u_0(x)v_0(x)-0.031 f_6(x) P_f'\nonumber \\ &&
 -0.044 f_6(x) P_0'(x)-0.075 f_3(x) u_0'(x) \nonumber \\&& -0.029 f_6(x)u_0(x)u_0'(x)-0.029 f_8(x) v_0(x) u_0'(x)
 \nonumber \\ && +0.033 f_4(x)u_0''(x) \nonumber \end{eqnarray}
\begin{eqnarray}
Line\;4&& \nonumber \\
u_4(x)&=& -0.121 f_5(x) P_f(x)+0.121 f_5(x) P_0(x)
\nonumber \\ &&+0.597 u_0(x)-0.084 f_2(x)u_0(x) +0.064 f_5(x) u_0(x)^2
\nonumber \\ && +0.064 f_7(x) u_0(x)v_0(x)-0.037 f_6(x) P_f'\nonumber \\ &&
 -0.046 f_6(x) P_0'(x)-0.084 f_3(x) u_0'(x) \nonumber \\ &&-0.029 f_6(x)u_0(x)u_0'(x)-0.029 f_8(x) v_0(x) u_0'(x)
 \nonumber \\ && +0.034 f_4(x)u_0''(x) \nonumber \end{eqnarray}

\begin{eqnarray}
Line\;5&& \nonumber \\
u_5(x)&=& -0.126 f_5(x) P_f(x)+0.126 f_5(x) P_0(x)
\nonumber \\ &&+0.497 u_0(x)-0.086 f_2(x)u_0(x) +0.062 f_5(x) u_0(x)^2
\nonumber \\ && +0.062 f_7(x) u_0(x)v_0(x)-0.041 f_6(x) P_f'\nonumber \\ &&
 -0.044 f_6(x) P_0'(x)-0.086 f_3(x) u_0'(x) \nonumber \\ &&-0.026 f_6(x)u_0(x)u_0'(x)-0.026 f_8(x) v_0(x) u_0'(x)
 \nonumber \\ && +0.032 f_4(x)u_0''(x) \nonumber \end{eqnarray}

\begin{eqnarray}
Line\;6&& \nonumber \\
u_6(x)&=& -0.121 f_5(x) P_f(x)+0.121 f_5(x) P_0(x)
\nonumber \\ &&+0.397 u_0(x)-0.080 f_2(x)u_0(x) +0.056 f_5(x) u_0(x)^2
\nonumber \\ && +0.056 f_7(x) u_0(x)v_0(x)-0.041 f_6(x) P_f'\nonumber \\ &&
 -0.031 f_6(x) P_0'(x)-0.069 f_3(x) u_0'(x) \nonumber \\ &&-0.017 f_6(x)u_0(x)u_0'(x)-0.017 f_8(x) v_0(x) u_0'(x)
 \nonumber \\ && +0.022 f_4(x)u_0''(x) \nonumber \end{eqnarray}

\begin{eqnarray}
Line\;7&& \nonumber \\
u_7(x)&=& -0.105 f_5(x) P_f(x)+0.105 f_5(x) P_0(x)
\nonumber \\ &&+0.297 u_0(x)-0.069 f_2(x)u_0(x) +0.045 f_5(x) u_0(x)^2
\nonumber \\ && +0.045 f_7(x) u_0(x)v_0(x)-0.037 f_6(x) P_f'\nonumber \\ &&
 -0.031 f_6(x) P_0'(x)-0.069 f_3(x) u_0'(x) \nonumber \\ &&-0.017 f_6(x)u_0(x)u_0'(x)-0.017 f_8(x) v_0(x) u_0'(x)
 \nonumber \\ && +0.022 f_4(x)u_0''(x) \nonumber \end{eqnarray}

\begin{eqnarray}
Line\;8&& \nonumber \\
u_8(x)&=& -0.080 f_5(x) P_f(x)+0.080 f_5(x) P_0(x)
\nonumber \\ &&+0.198 u_0(x)-0.051 f_2(x)u_0(x) +0.032 f_5(x) u_0(x)^2
\nonumber \\ && +0.032 f_7(x) u_0(x)v_0(x)-0.029 f_6(x) P_f'\nonumber \\ &&
 -0.021 f_6(x) P_0'(x)-0.051 f_3(x) u_0'(x) \nonumber \\ &&-0.011 f_6(x)u_0(x)u_0'(x)-0.011 f_8(x) v_0(x) u_0'(x)
 \nonumber \\ && +0.015 f_4(x)u_0''(x) \nonumber \end{eqnarray}
\begin{eqnarray}
Line\;9&& \nonumber \\
u_9(x)&=& -0.045 f_5(x) P_f(x)+0.045 f_5(x) P_0(x)
\nonumber \\ &&+0.099 u_0(x)-0.028 f_2(x)u_0(x) +0.016 f_5(x) u_0(x)^2
\nonumber \\ && +0.016 f_7(x) u_0(x)v_0(x)-0.017 f_6(x) P_f'\nonumber \\ &&
 -0.022 f_6(x) P_0'(x)-0.012 f_3(x) u_0'(x) \nonumber \\ &&-0.006 f_6(x)u_0(x)u_0'(x)-0.006 f_8(x) v_0(x) u_0'(x)
 \nonumber \\ && +0.008 f_4(x)u_0''(x) \nonumber \end{eqnarray}

For the field of velocity $v$, we have obtained  the following expressions for the lines:

\begin{eqnarray}
Line\;1&& \nonumber \\
v_1(x)&=& -0.045 f_7(x) P_f(x)+0.045 f_7(x) P_0(x)
 \nonumber \\ && +0.899 v_0(x)-0.034 f_2(x)v_0(x)+0.029 f_5(x) u_0(x)v_0(x)
 \nonumber \\ &&+0.029 f_7(x)v_0(x)^2-0.011 f_8(x)P_f'(x) \nonumber \\ &&
 -0.022 f_8(x)P_0'(x)-0.034 f_3(x) v_0'(x)
 \nonumber \\ && -0.016 f_6(x)u_0(x)v_0'(x) -0.016 f_8(x)v_0(x)v_0'(x)
  \nonumber \\ && +0.018f_4(x)v_0''(x) \nonumber
\end{eqnarray}
\begin{eqnarray}
Line\;2&& \nonumber \\
v_2(x)&=& -0.081 f_7(x) P_f(x)+0.081 f_7(x) P_0(x)
 \nonumber \\ && +0.799 v_0(x)-0.059 f_2(x)v_0(x)+0.048 f_5(x) u_0(x)v_0(x)
 \nonumber \\ &&+0.048 f_7(x)v_0(x)^2-0.022 f_8(x)P_f'(x) \nonumber \\ &&
 -0.036 f_8(x)P_0'(x)-0.059 f_3(x) v_0'(x)
 \nonumber \\ && -0.025 f_6(x)u_0(x)v_0'(x) -0.025 f_8(x)v_0(x)v_0'(x)
  \nonumber \\ && +0.028f_4(x)v_0''(x) \nonumber
  \end{eqnarray}
 \begin{eqnarray}
Line\;3&& \nonumber \\
v_3(x)&=& -0.106 f_7(x) P_f(x)+0.106 f_7(x) P_0(x)
 \nonumber \\ && +0.698 v_0(x)-0.075 f_2(x)v_0(x)+0.060 f_5(x) u_0(x)v_0(x)
 \nonumber \\ &&+0.060 f_7(x)v_0(x)^2-0.031 f_8(x)P_f'(x) \nonumber \\ &&
 -0.044 f_8(x)P_0'(x)-0.075 f_3(x) v_0'(x)
 \nonumber \\ && -0.029 f_6(x)u_0(x)v_0'(x) -0.029 f_8(x)v_0(x)v_0'(x)
  \nonumber \\ && +0.033 f_4(x)v_0''(x) \nonumber
  \end{eqnarray}
 \begin{eqnarray}
Line\;4&& \nonumber \\
v_4(x)&=& -0.121 f_7(x) P_f(x)+0.121 f_7(x) P_0(x)
 \nonumber \\ && +0.597 v_0(x)-0.084 f_2(x)v_0(x)+0.064 f_5(x) u_0(x)v_0(x)
 \nonumber \\ &&+0.064 f_7(x)v_0(x)^2-0.037 f_8(x)P_f'(x) \nonumber \\ &&
 -0.046 f_8(x)P_0'(x)-0.084 f_3(x) v_0'(x)
 \nonumber \\ && -0.029 f_6(x)u_0(x)v_0'(x) -0.029 f_8(x)v_0(x)v_0'(x)
  \nonumber \\ && +0.034f_4(x)v_0''(x) \nonumber
  \end{eqnarray}
\begin{eqnarray}
Line\;5&& \nonumber \\
v_5(x)&=& -0.126 f_7(x) P_f(x)+0.126 f_7(x) P_0(x)
 \nonumber \\ && +0.497 v_0(x)-0.086 f_2(x)v_0(x)+0.062 f_5(x) u_0(x)v_0(x)
 \nonumber \\ &&+0.062 f_7(x)v_0(x)^2-0.041 f_8(x)P_f'(x) \nonumber \\ &&
 -0.044 f_8(x)P_0'(x)-0.086 f_3(x) v_0'(x)
 \nonumber \\ && -0.026 f_6(x)u_0(x)v_0'(x) -0.026 f_8(x)v_0(x)v_0'(x)
  \nonumber \\ && +0.032 f_4(x)v_0''(x) \nonumber
  \end{eqnarray}
 \begin{eqnarray}
Line\;6&& \nonumber \\
v_6(x)&=& -0.121 f_7(x) P_f(x)+0.121 f_7(x) P_0(x)
 \nonumber \\ && +0.397 v_0(x)-0.080 f_2(x)v_0(x)+0.056 f_5(x) u_0(x)v_0(x)
 \nonumber \\ &&+0.056 f_7(x)v_0(x)^2-0.022 f_8(x)P_f'(x) \nonumber \\ &&
 -0.041 f_8(x)P_0'(x)-0.080 f_3(x) v_0'(x)
 \nonumber \\ && -0.022 f_6(x)u_0(x)v_0'(x) -0.022 f_8(x)v_0(x)v_0'(x)
  \nonumber \\ && +0.028f_4(x)v_0''(x) \nonumber
  \end{eqnarray}
 \begin{eqnarray}
Line\;7&& \nonumber \\
v_7(x)&=& -0.105 f_7(x) P_f(x)+0.105 f_7(x) P_0(x)
 \nonumber \\ && +0.297 v_0(x)-0.069 f_2(x)v_0(x)+0.045 f_5(x) u_0(x)v_0(x)
 \nonumber \\ &&+0.045 f_7(x)v_0(x)^2-0.037 f_8(x)P_f'(x) \nonumber \\ &&
 -0.031 f_8(x)P_0'(x)-0.069 f_3(x) v_0'(x)
 \nonumber \\ && -0.017 f_6(x)u_0(x)v_0'(x) -0.017 f_8(x)v_0(x)v_0'(x)
  \nonumber \\ && +0.022 f_4(x)v_0''(x) \nonumber
  \end{eqnarray}
\begin{eqnarray}
Line\;8&& \nonumber \\
v_8(x)&=& -0.080 f_7(x) P_f(x)+0.080 f_7(x) P_0(x)
 \nonumber \\ && +0.198 v_0(x)-0.051 f_2(x)v_0(x)+0.032 f_5(x) u_0(x)v_0(x)
 \nonumber \\ &&+0.032 f_7(x)v_0(x)^2-0.029 f_8(x)P_f'(x) \nonumber \\ &&
 -0.021 f_8(x)P_0'(x)-0.051 f_3(x) v_0'(x)
 \nonumber \\ && -0.011 f_6(x)u_0(x)v_0'(x) -0.011 f_8(x)v_0(x)v_0'(x)
  \nonumber \\ && +0.015 f_4(x)v_0''(x) \nonumber
  \end{eqnarray}
\begin{eqnarray}
Line\;9&& \nonumber \\
v_9(x)&=& -0.045 f_7(x) P_f(x)+0.045 f_7(x) P_0(x)
 \nonumber \\ && +0.099 v_0(x)-0.028 f_2(x)v_0(x)+0.016 f_5(x) u_0(x)v_0(x)
 \nonumber \\ &&+0.016 f_7(x)v_0(x)^2-0.017 f_8(x)P_f'(x) \nonumber \\ &&
 -0.011 f_8(x)P_0'(x)-0.029 f_3(x) v_0'(x)
 \nonumber \\ && -0.057 f_6(x)u_0(x)v_0'(x) -0.057 f_8(x)v_0(x)v_0'(x)
  \nonumber \\ && +0.008 f_4(x)v_0''(x) \nonumber
  \end{eqnarray}
Finally, for the field of pressure $P$, we have obtained the following lines:

\begin{eqnarray}
Line\;1&& \nonumber \\
P_1(x)&=& 0.101 P_f(x)+0.034 f_2(x) P_f(x) \nonumber \\ &&
+0.899 P_0(x)-0.034 f_2(x) P_0(x)+0.046 h_3(x) f_5(x)^2 u_0(x)^2
\nonumber \\ && +0.092 h_3(x)f_5(x)f_7(x) u_0(x)v_0(x)+0.046  h_3(x)f_7(x)^2v_0(x)^2
\nonumber \\ &&+0.034 f_3(x) P_f'(x)-0.034 f_3(x) P_0'(x)
 \nonumber \\ &&-0.045  h_3(x)f_5(x) f_6(x)u_0(x) u_0'(x)-0.045  h_3(x)f_5(x)f_8(x)v_0(x)u_0'(x)
 \nonumber \\ &&+0.014  h_3(x)f_6(x)^2u_0'(x)^2-0.045  h_3(x)f_6(x) f_7(x) u_0(x) v_0'(x)
 \nonumber \\ &&-0.045 h_3(x) f_7(x) f_8(x) v_0(x)v_0'(x)+0.027 h_3(x) f_6(x)f_8(x) u_0'(x)v_0'(x)
  \nonumber \\ && +0.014 h_3(x) f_8(x)^2 v_0'(x)^2+0.008 f_4(x) P_f''(x)
  \nonumber \\ &&+0.018 f_4(x)P_0''(x)
  \nonumber \end{eqnarray}
\begin{eqnarray}
Line\;2&& \nonumber \\
P_2(x)&=& 0.201 P_f(x)+0.059 f_2(x) P_f(x) \nonumber \\ &&
+0.799 P_0(x)-0.059 f_2(x) P_0(x)+0.082  h_3(x)f_5(x)^2 u_0(x)^2
\nonumber \\&& +0.163  h_3(x)f_5(x)f_7(x) u_0(x)v_0(x)+0.082  h_3(x)f_7(x)^2v_0(x)^2
\nonumber \\ &&+0.059 f_3(x) P_f'(x)-0.059 f_3(x) P_0'(x)
 \nonumber \\ &&-0.074  h_3(x)f_5(x) f_6(x)u_0(x) u_0'(x)-0.074  h_3(x)f_5(x)f_8(x)v_0(x)u_0'(x)
 \nonumber \\ &&+0.020  h_3(x)f_6(x)^2u_0'(x)^2-0.074  h_3(x)f_6(x) f_7(x) u_0(x) v_0'(x)
 \nonumber \\ &&-0.074  h_3(x)f_7(x) f_8(x) v_0(x)v_0'(x)+0.041  h_3(x)f_6(x)f_8(x) u_0'(x)v_0'(x)
  \nonumber \\ && +0.020  h_3(x)f_8(x)^2 v_0'(x)^2+0.015 f_4(x) P_f''(x)
  \nonumber \\ &&+0.028 f_4(x)P_0''(x)
  \nonumber \end{eqnarray}

\begin{eqnarray}
Line\;3&& \nonumber \\
P_3(x)&=& 0.302 P_f(x)+0.075f_2(x) P_f(x) \nonumber \\ &&
+0.698 P_0(x)-0.075 f_2(x) P_0(x)+0.107 h_3(x) f_5(x)^2 u_0(x)^2
\nonumber \\&& +0.214  h_3(x)f_5(x)f_7(x) u_0(x)v_0(x)+0.107  h_3(x)f_7(x)^2v_0(x)^2
\nonumber \\ &&+0.075 f_3(x) P_f'(x)-0.075 f_3(x) P_0'(x)
 \nonumber \\ &&-0.089  h_3(x)f_5(x) f_6(x)u_0(x) u_0'(x)-0.089  h_3(x)f_5(x)f_8(x)v_0(x)u_0'(x)
 \nonumber \\ &&+0.023  h_3(x)f_6(x)^2u_0'(x)^2-0.089  h_3(x)f_6(x) f_7(x) u_0(x) v_0'(x)
 \nonumber \\ &&-0.089  h_3(x)f_7(x) f_8(x) v_0(x)v_0'(x)+0.045  h_3(x)f_6(x)f_8(x) u_0'(x)v_0'(x)
  \nonumber \\ && +0.023  h_3(x)f_8(x)^2 v_0'(x)^2+0.021 f_4(x) P_f''(x)
  \nonumber \\ &&+0.033 f_4(x)P_0''(x)
  \nonumber \end{eqnarray}

\begin{eqnarray}
Line\;4&& \nonumber \\
P_4(x)&=& 0.403 P_f(x)+0.084 f_2(x) P_f(x) \nonumber \\ &&
+0.597 P_0(x)-0.084 f_2(x) P_0(x)+0.122  h_3(x)f_5(x)^2 u_0(x)^2
\nonumber \\ &&+0.245  h_3(x)f_5(x)f_7(x) u_0(x)v_0(x)+0.122  h_3(x)f_7(x)^2v_0(x)^2
\nonumber \\ &&+0.084 f_3(x) P_f'(x)-0.084 f_3(x) P_0'(x)
 \nonumber \\ &&-0.094  h_3(x)f_5(x) f_6(x)u_0(x) u_0'(x)-0.094  h_3(x)f_5(x)f_8(x)v_0(x)u_0'(x)
 \nonumber \\ &&+0.022  h_3(x)f_6(x)^2u_0'(x)^2-0.094  h_3(x)f_6(x) f_7(x) u_0(x) v_0'(x)
 \nonumber \\ &&-0.094  h_3(x)f_7(x) f_8(x) v_0(x)v_0'(x)+0.045  h_3(x)f_6(x)f_8(x) u_0'(x)v_0'(x)
  \nonumber \\ && +0.022  h_3(x)f_8(x)^2 v_0'(x)^2+0.025 f_4(x) P_f''(x)
  \nonumber \\ &&+0.034 f_4(x)P_0''(x)
  \nonumber \end{eqnarray}
\begin{eqnarray}
Line\;5&& \nonumber \\
P_5(x)&=& 0.503 P_f(x)+0.086 f_2(x) P_f(x) \nonumber \\ &&
+0.497 P_0(x)-0.086 f_2(x) P_0(x)+0.127  h_3(x)f_5(x)^2 u_0(x)^2
\nonumber \\&& +0.255  h_3(x)f_5(x)f_7(x) u_0(x)v_0(x)+0.127  h_3(x)f_7(x)^2v_0(x)^2
\nonumber \\ &&+0.086 f_3(x) P_f'(x)-0.086 f_3(x) P_0'(x)
 \nonumber \\ &&-0.089  h_3(x)f_5(x) f_6(x)u_0(x) u_0'(x)-0.089  h_3(x)f_5(x)f_8(x)v_0(x)u_0'(x)
 \nonumber \\ &&+0.020  h_3(x)f_6(x)^2u_0'(x)^2-0.089  h_3(x)f_6(x) f_7(x) u_0(x) v_0'(x)
 \nonumber \\ &&-0.089  h_3(x)f_7(x) f_8(x) v_0(x)v_0'(x)+0.040  h_3(x)f_6(x)f_8(x) u_0'(x)v_0'(x)
  \nonumber \\ && +0.020  h_3(x)f_8(x)^2 v_0'(x)^2+0.026 f_4(x) P_f''(x)
  \nonumber \\ &&+0.032 f_4(x)P_0''(x)
  \nonumber \end{eqnarray}
\begin{eqnarray}
Line\;6&& \nonumber \\
P_6(x)&=& 0.603 P_f(x)+0.080 f_2(x) P_f(x) \nonumber \\ &&
+0.397 P_0(x)-0.080 f_2(x) P_0(x)+0.122 h_3(x) f_5(x)^2 u_0(x)^2
\nonumber \\&& +0.244  h_3(x)f_5(x)f_7(x) u_0(x)v_0(x)+0.122 h_3(x) f_7(x)^2v_0(x)^2
\nonumber \\ &&+0.080 f_3(x) P_f'(x)-0.080 f_3(x) P_0'(x)
 \nonumber \\ &&-0.079  h_3(x)f_5(x) f_6(x)u_0(x) u_0'(x)-0.079  h_3(x)f_5(x)f_8(x)v_0(x)u_0'(x)
 \nonumber \\ &&+0.017  h_3(x)f_6(x)^2u_0'(x)^2-0.079 h_3(x) f_6(x) f_7(x) u_0(x) v_0'(x)
 \nonumber \\ &&-0.079  h_3(x)f_7(x) f_8(x) v_0(x)v_0'(x)+0.033  h_3(x)f_6(x)f_8(x) u_0'(x)v_0'(x)
  \nonumber \\ && +0.017  h_3(x)f_8(x)^2 v_0'(x)^2+0.026 f_4(x) P_f''(x)
  \nonumber \\ &&+0.028 f_4(x)P_0''(x)
  \nonumber \end{eqnarray}
\begin{eqnarray}
Line\;7&& \nonumber \\
P_7(x)&=& 0.703 P_f(x)+0.069 f_2(x) P_f(x) \nonumber \\ &&
+0.297 P_0(x)-0.069 f_2(x) P_0(x)+0.107 h_3(x) f_5(x)^2 u_0(x)^2
\nonumber \\ &&+0.213  h_3(x)f_5(x)f_7(x) u_0(x)v_0(x)+0.107  h_3(x)f_7(x)^2v_0(x)^2
\nonumber \\ &&+0.069 f_3(x) P_f'(x)-0.069 f_3(x) P_0'(x)
 \nonumber \\ &&-0.063 f_5(x) f_6(x)u_0(x) u_0'(x)-0.063 f_5(x)f_8(x)v_0(x)u_0'(x)
 \nonumber \\ &&+0.013  h_3(x)f_6(x)^2u_0'(x)^2-0.063  h_3(x)f_6(x) f_7(x) u_0(x) v_0'(x)
 \nonumber \\ &&-0.063  h_3(x)f_7(x) f_8(x) v_0(x)v_0'(x)+0.025  h_3(x)f_6(x)f_8(x) u_0'(x)v_0'(x)
  \nonumber \\ && +0.013  h_3(x)f_8(x)^2 v_0'(x)^2+0.023 f_4(x) P_f''(x)
  \nonumber \\ &&+0.022 f_4(x)P_0''(x)
  \nonumber \end{eqnarray}

  \begin{eqnarray}
Line\;8&& \nonumber \\
P_8(x)&=& 0.802 P_f(x)+0.051 f_2(x) P_f(x) \nonumber \\ &&
+0.198 P_0(x)-0.051 f_2(x) P_0(x)+0.081  h_3(x)f_5(x)^2 u_0(x)^2
\nonumber \\ &&+0.162 h_3(x) f_5(x)f_7(x) u_0(x)v_0(x)+0.081  h_3(x)f_7(x)^2v_0(x)^2
\nonumber \\ &&+0.051 f_3(x) P_f'(x)-0.051 f_3(x) P_0'(x)
 \nonumber \\ &&-0.043  h_3(x)f_5(x) f_6(x)u_0(x) u_0'(x)-0.043 h_3(x) f_5(x)f_8(x)v_0(x)u_0'(x)
 \nonumber \\ &&+0.009  h_3(x)f_6(x)^2u_0'(x)^2-0.043  h_3(x)f_6(x) f_7(x) u_0(x) v_0'(x)
 \nonumber \\ &&-0.043  h_3(x)f_7(x) f_8(x) v_0(x)v_0'(x)+0.017  h_3(x)f_6(x)f_8(x) u_0'(x)v_0'(x)
  \nonumber \\ && +0.009  h_3(x)f_8(x)^2 v_0'(x)^2+0.018 f_4(x) P_f''(x)
  \nonumber \\ &&+0.015 f_4(x)P_0''(x)
  \nonumber \end{eqnarray}
  \begin{eqnarray}
Line\;9&& \nonumber \\
P_9(x)&=& 0.901 P_f(x)+0.028 f_2(x) P_f(x) \nonumber \\ &&
+0.099 P_0(x)-0.028 f_2(x) P_0(x)+0.045  h_3(x)f_5(x)^2 u_0(x)^2
\nonumber \\ && +0.191  h_3(x)f_5(x)f_7(x) u_0(x)v_0(x)+0.045  h_3(x)f_7(x)^2v_0(x)^2
\nonumber \\ &&+0.028 f_3(x) P_f'(x)-0.028 f_3(x) P_0'(x)
 \nonumber \\ &&-0.022  h_3(x)f_5(x) f_6(x)u_0(x) u_0'(x)-0.022  h_3(x)f_5(x)f_8(x)v_0(x)u_0'(x)
 \nonumber \\ &&+0.004 h_3(x) f_6(x)^2u_0'(x)^2-0.022  h_3(x)f_6(x) f_7(x) u_0(x) v_0'(x)
 \nonumber \\ &&-0.022  h_3(x)f_7(x) f_8(x) v_0(x)v_0'(x)+0.009 h_3(x) f_6(x)f_8(x) u_0'(x)v_0'(x)
  \nonumber \\ && +0.004  h_3(x)f_8(x)^2 v_0'(x)^2+0.010 f_4(x) P_f''(x)
  \nonumber \\ &&+0.008 f_4(x)P_0''(x)
  \nonumber \end{eqnarray}

\subsection{ Numerical examples through the generalized method of lines}
 We consider some examples in which  $$\Omega=\{(r,\theta)\;|\;   1\leq r\leq 2,\; 0\leq \theta \leq 2\pi\},$$
$$\partial \Omega_0=\{(1,\theta)\;|\;0\leq \theta \leq 2\pi\},$$ and
$$\partial \Omega_1=\{(2,\theta)\;|\;0\leq \theta \leq 2\pi\}.$$
For such cases, the boundary conditions are

%$$u=-3.0\sin(\theta), \;\; v=3.0\cos(\theta),\; \text{ on } \partial \Omega_0,$$
%$$u=v=0, \;P=2.0 \text{ on } \partial \Omega_1,$$
$$u=u_0(\theta), \;\; v=v_0(\theta),\; \text{ on } \partial \Omega_0,$$
$$u=v=0, \text{ on } \partial \Omega_1,$$
so that in these examples we do not have boundary conditions for the pressure.

Through the generalized method of lines, neglecting gravity effects and truncating the series up to the terms in $d^2$,
 where $d=1/N$ is
the mesh thickness concerning the discretization in $r$, we present numerical results for the following
 approximation of the Navier-Stokes system,
 \begin{gather}\label{912} \left\{
\begin{array}{ll}
 \nu \nabla^2u-u \partial_x u-v \partial_y u- \partial_x P =0, & \text{ in }  \Omega,
 \\ \\
 \nu \nabla^2 v-u \partial_x v-v \partial_y v- \partial_y P=0, & \text{ in } \Omega,
\\ \\
\varepsilon \nabla^2P + \partial_x u+ \partial_y v=0, & \text{ in } \Omega,
\end{array} \right.\end{gather}
where $\varepsilon>0$ is a very small parameter. We highlight, since $\varepsilon>0$ must be very small, the results obtained
through the generalized of lines, as indicated in the last section, may have a relevant error. Anyway, we may use such a procedure
to obtain the general line expressions, which we expect to be analytically suitable to obtain a numerical result by calculating
numerically the concerning coefficients for such lines, through a minimization of the $L^2$ norm of equation errors, in
a finite differences context.

Thus, from such a method, the general expression for the velocity and pressure fields on the line $n$, are given by
(here $x$ stands for $\theta$ and $P_0$ must be calculated numerically in the optimization process):
\begin{eqnarray}
u_n(x)&=& a_1[n]\cos(x)+a_2[n] u_0(x)+a_3[n]\cos[x]u_0(x)^2 \nonumber \\ &&
a_4[n]\sin(x)u_0(x)v_0(x)+a_5[n]\sin(x)u_0(x)u_0'(x)+a_6[n]\cos(x) v_0(x)u_0'(x)
\nonumber \\ && a_7[n]u_0''(x)+a_8[n] \sin(x)+a_9[n]\sin(x)P_0(x) \nonumber \\ &&
a_{10}[n]P_0(x)\cos(x)+a_{11}[n],
\end{eqnarray}
\begin{eqnarray}
v_n(x)&=& b_1[n]\sin(x)+b_2[n] v_0(x)+b_3[n]\sin[x]v_0(x)^2 \nonumber \\ &&
b_4[n]\cos(x)u_0(x)v_0(x)+b_5[n]\sin(x)v_0'(x)u_0(x)+b_6[n]\cos(x) v_0(x)u_0'(x)
\nonumber \\ && b_7[n]v_0''(x)+b_8[n] \cos(x)+b_9[n]\sin(x)P_0(x) \nonumber \\ &&
b_{10}[n]P_0(x)\cos(x)+b_{11}[n],
\end{eqnarray}
\begin{eqnarray}
P_n(x)&=& c_1[n]+c_2[n] \cos(x)u_0(x)+c_3[n] \sin(x) v_0(x)
 +c_4[n]\sin(x) u_0'(x) \nonumber \\ &&+c_5[n]\cos(x)v_0'(x)+c_6[n]u_0''(x)+c_7[n]v_0''(x)
\nonumber \\ &&+c_9[n]P_0(x)+c_{10}[n] P_0''(x).
\end{eqnarray}

%\begin{eqnarray}
%u_n(x)= a_1[n]\cos(x) +a_2[n]\sin(x)+a_3[n]\cos(x)^3+a_4[n] \cos(x)\sin(x)^2\nonumber
%\end{eqnarray}
%\begin{eqnarray}
%v_n(x)= b_1[n]\cos(x) +b_2[n]\sin(x)+b_3[n]\sin(x)^3+b_4[n] \cos(x)^2\sin(x)\nonumber
%\end{eqnarray}
%\begin{eqnarray}
%P_n(x)= c_1[n]+c_2[n]\sin(x)^4+c_3[n]\cos(x)^4+c_4[n] \cos(x)^2\sin(x)^2.\nonumber
%\end{eqnarray}

For the first example, $u_0(x)=-1.5 \sin(x)$ and $v_0(x)=1.5 \cos(x).$
Denoting \begin{eqnarray}J(u,v,P)&=&\int_\Omega \left(\nu \nabla^2 u-u \partial_x u-v \partial_y u- \partial_x P\right)^2\;d\Omega \nonumber \\
&&+\int_\Omega \left(\nu \nabla^2 v-u \partial_x v-v \partial_y v- \partial_y P\right)^2\;d\Omega \nonumber \\
&&+\int_\Omega \left(\partial_x u+ \partial_y v\right)^2\;d\Omega,
\end{eqnarray} as, above mentioned, the coefficients $\{a_i[n]\}, \{b_i[n]\}, \{c_i[n]\}$ have been obtained through the numerical minimization of $J(\{u_n\},\{v_n\},\{P_n\})$, so that
 for the mesh in question, we have obtained
$$ \text{ For this first example: } J(\{u_n\},\{v_n\},\{P_n\}) \approx 9.23 \; 10^{-12} \text{ for } \nu=0.1,$$
We have plotted the fields  $u$, $v$ and $P$, for the lines $n=1,\;n=5,\;n=10,\;n=15\text{ and } n=19$, for a mesh $20 \times 150$ corresponding to $20$ lines.
Please, see the figures from \ref{figure4.1} to \ref{figure4.6} for the case $\nu=0.1$.
For all graphs, please consider units in $x$ to be multiplied by $2 \pi/150.$
\begin{figure}[!htb]
\begin{center}
	\includegraphics[height=4.0cm]{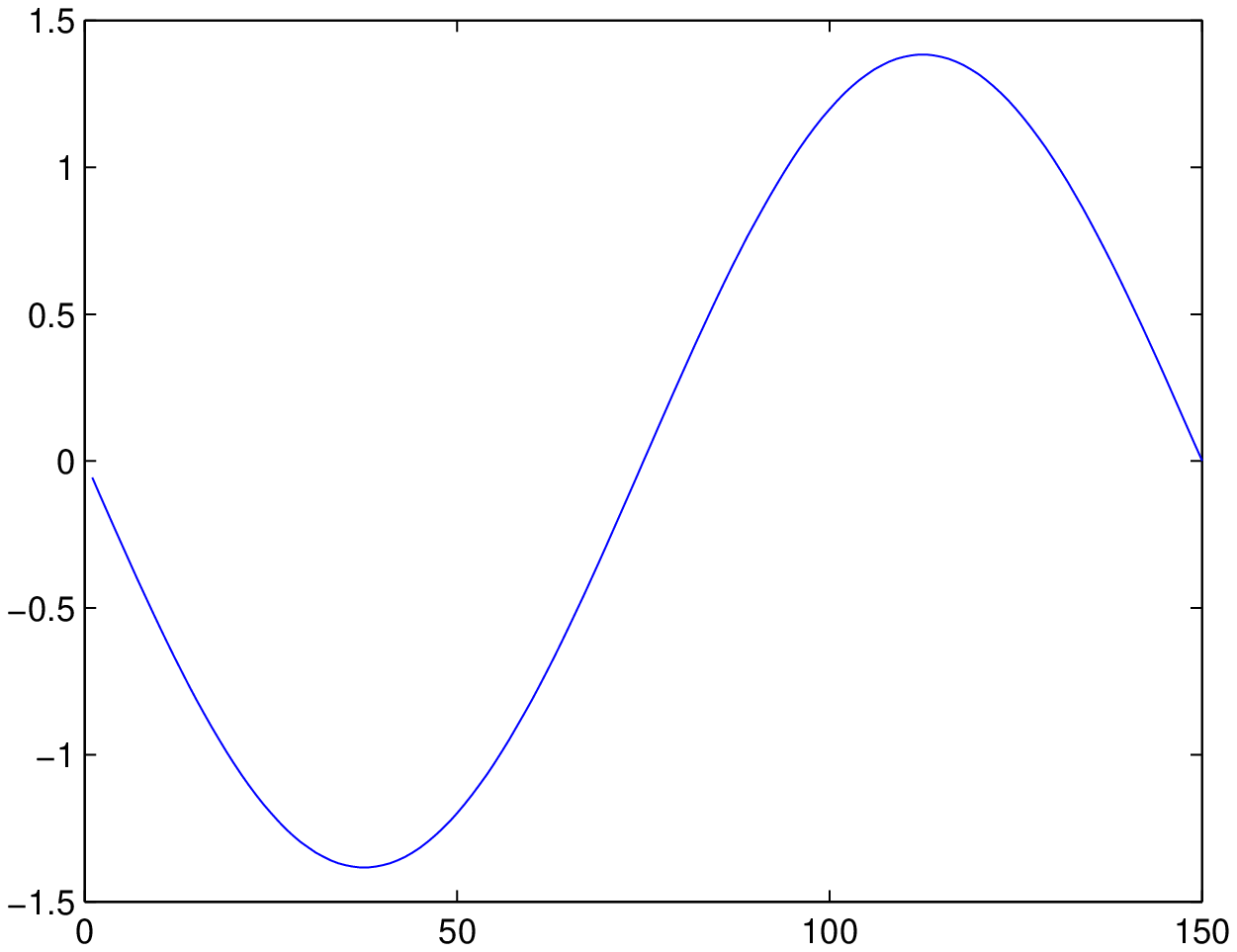}
    \includegraphics[height=4.0cm]{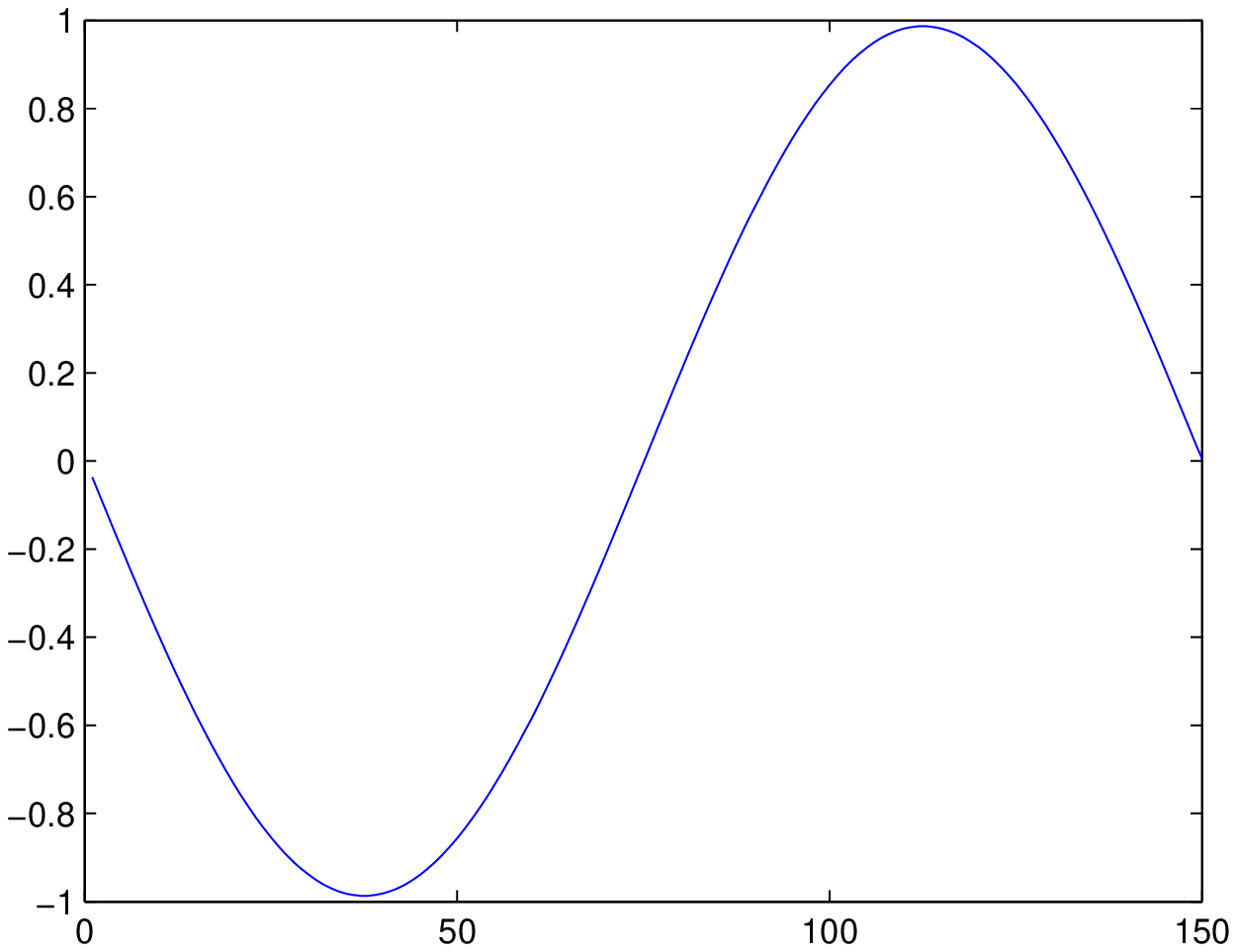}
    \includegraphics[height=4.0cm]{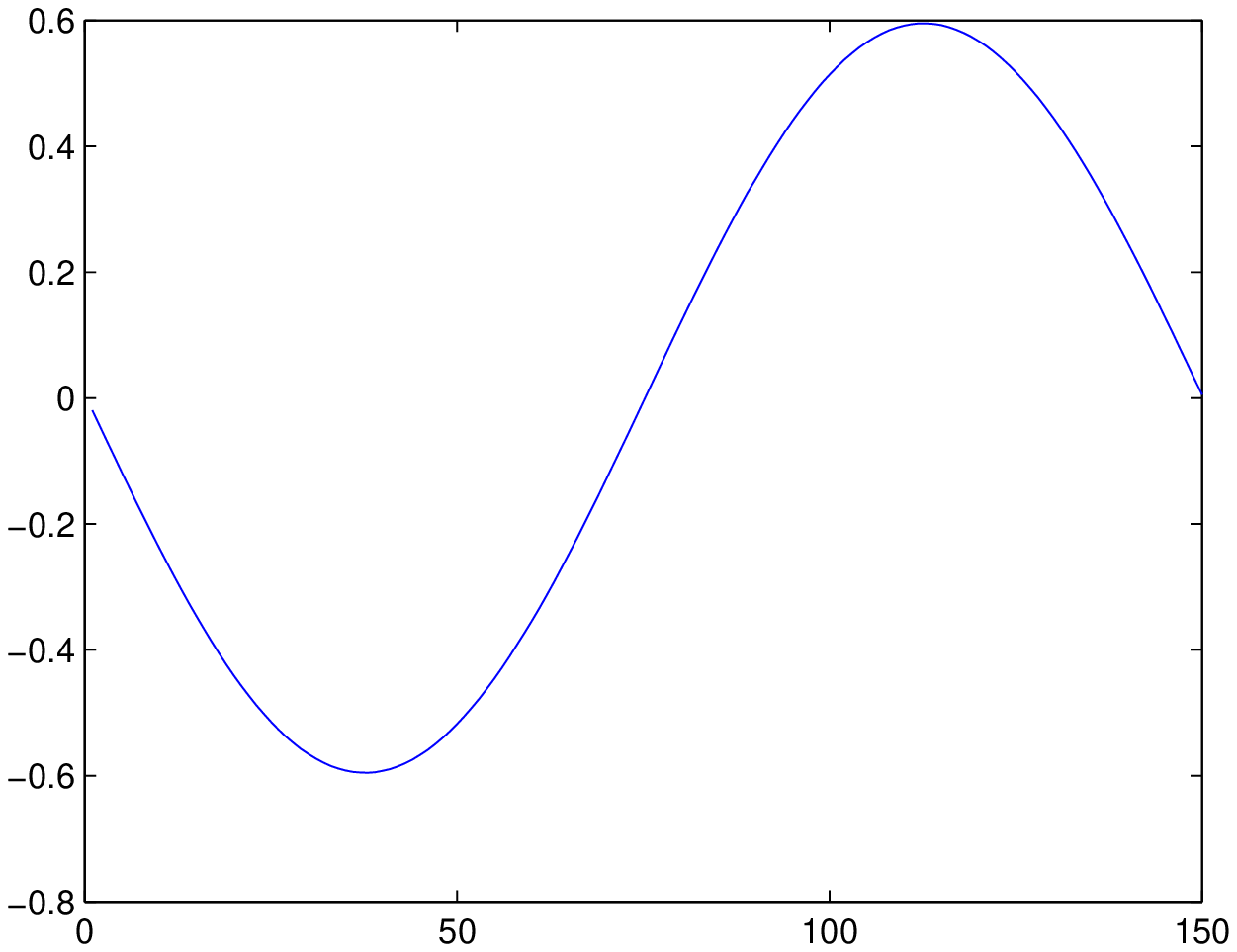}
    \vspace{-0.3cm}
\caption{\small{First example, from the left to the right, fields of velocity $u_1(x),\;u_5(x),\; u_{10}(x)$ for the lines $n=1,\;n=5$ and $n=10$.}} \label{figure4.1}
\end{center}
\end{figure}
\begin{figure}[!htb]
\begin{center}
	\includegraphics[height=4.0cm]{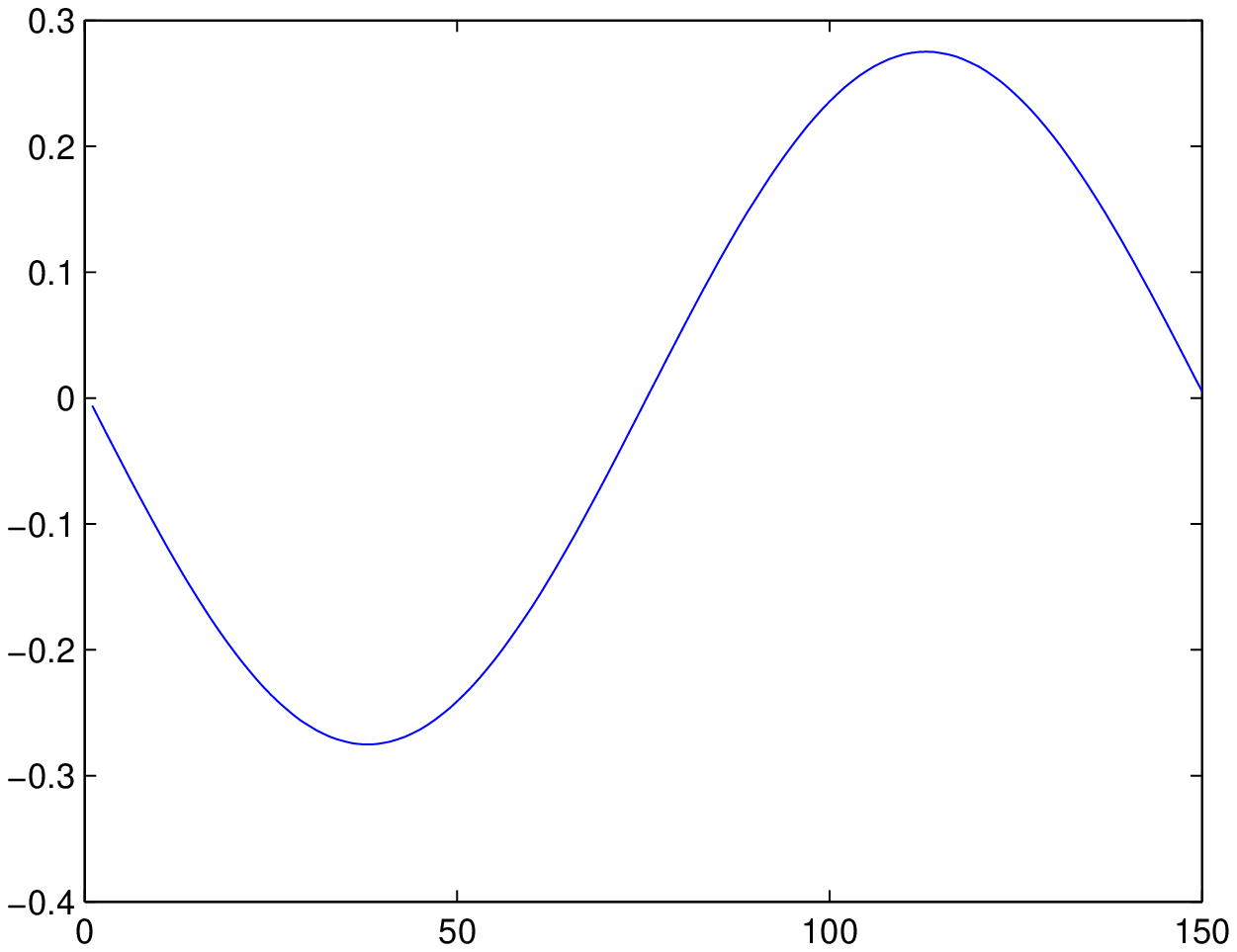}
    \includegraphics[height=4.0cm]{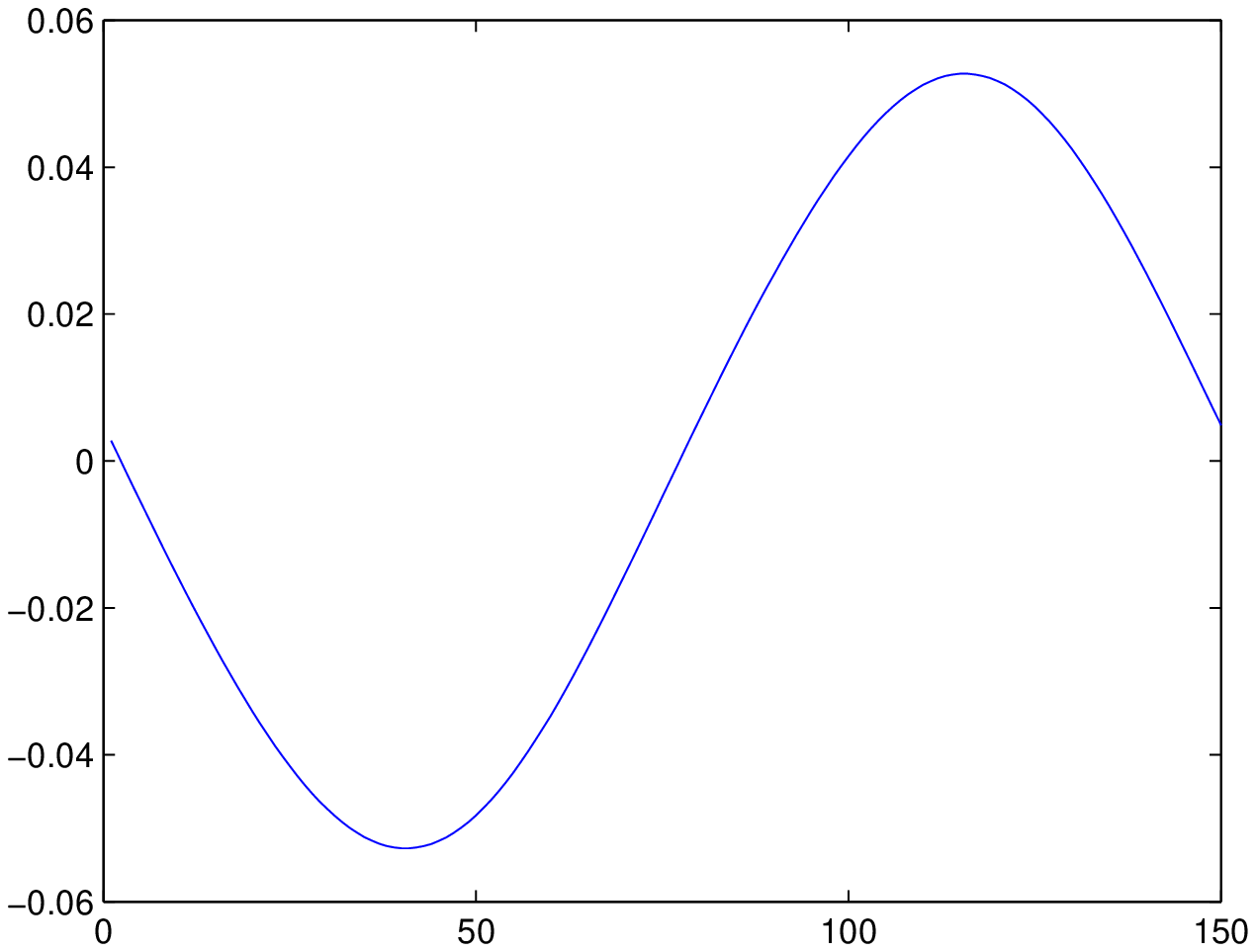}
    \vspace{-0.3cm}
\caption{\small{First example, from the left to the right, fields of velocity $u_{15}(x),\;u_{19}(x)$ for the lines $n=15$, and $n=19$.}} \label{figure4.2}
\end{center}
\end{figure}
\begin{figure}[!htb]
\begin{center}
	\includegraphics[height=4.0cm]{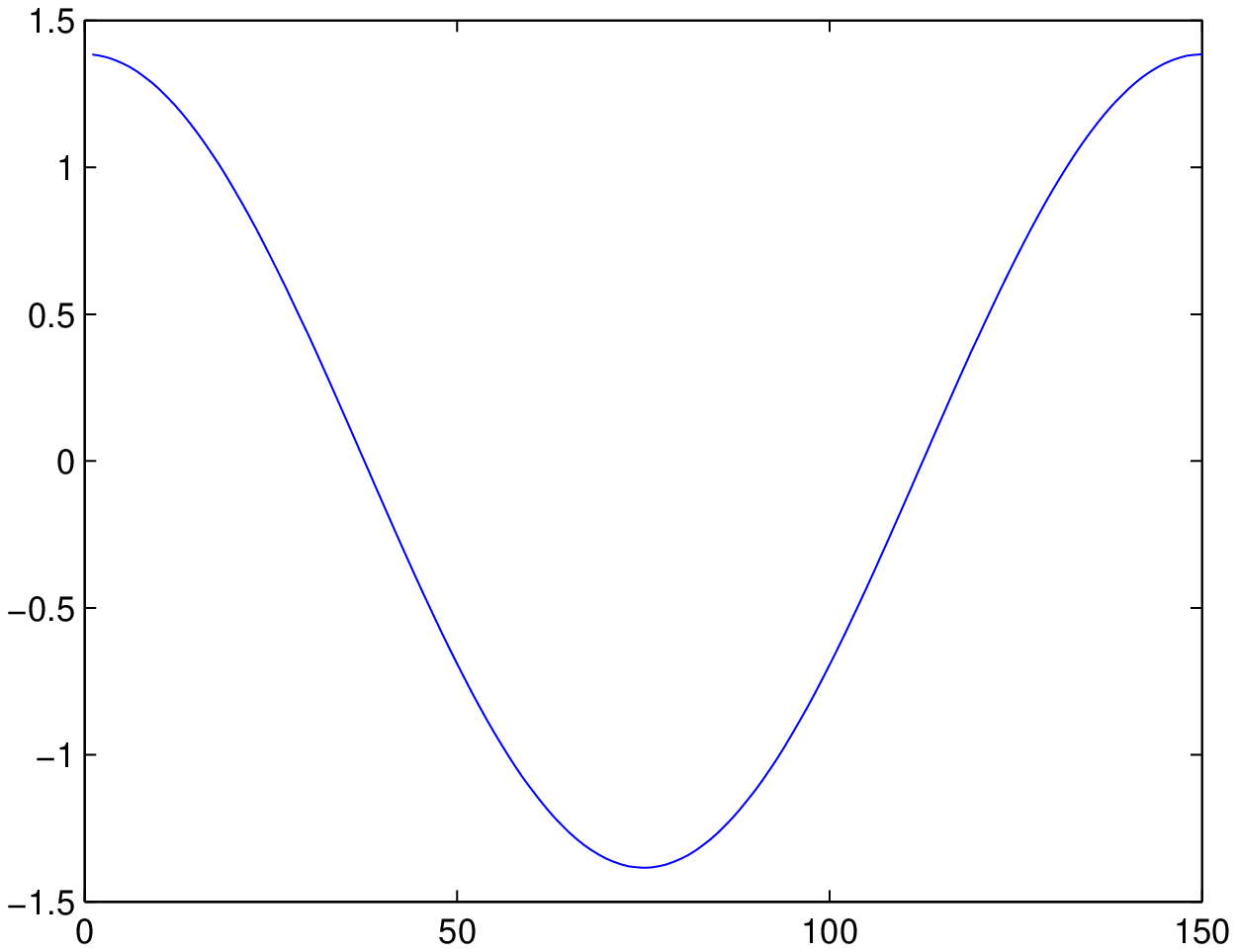}
    \includegraphics[height=4.0cm]{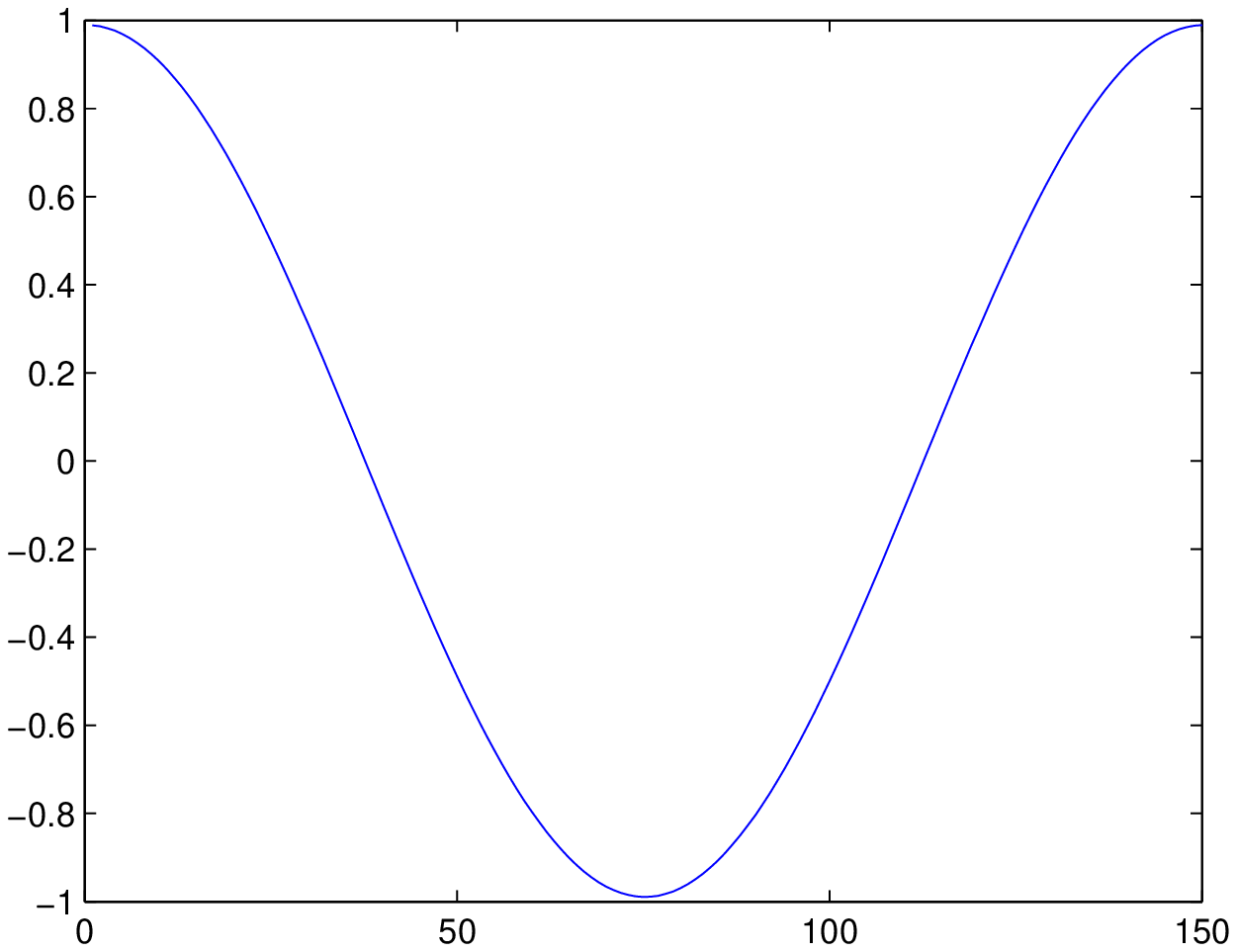}
    \includegraphics[height=4.0cm]{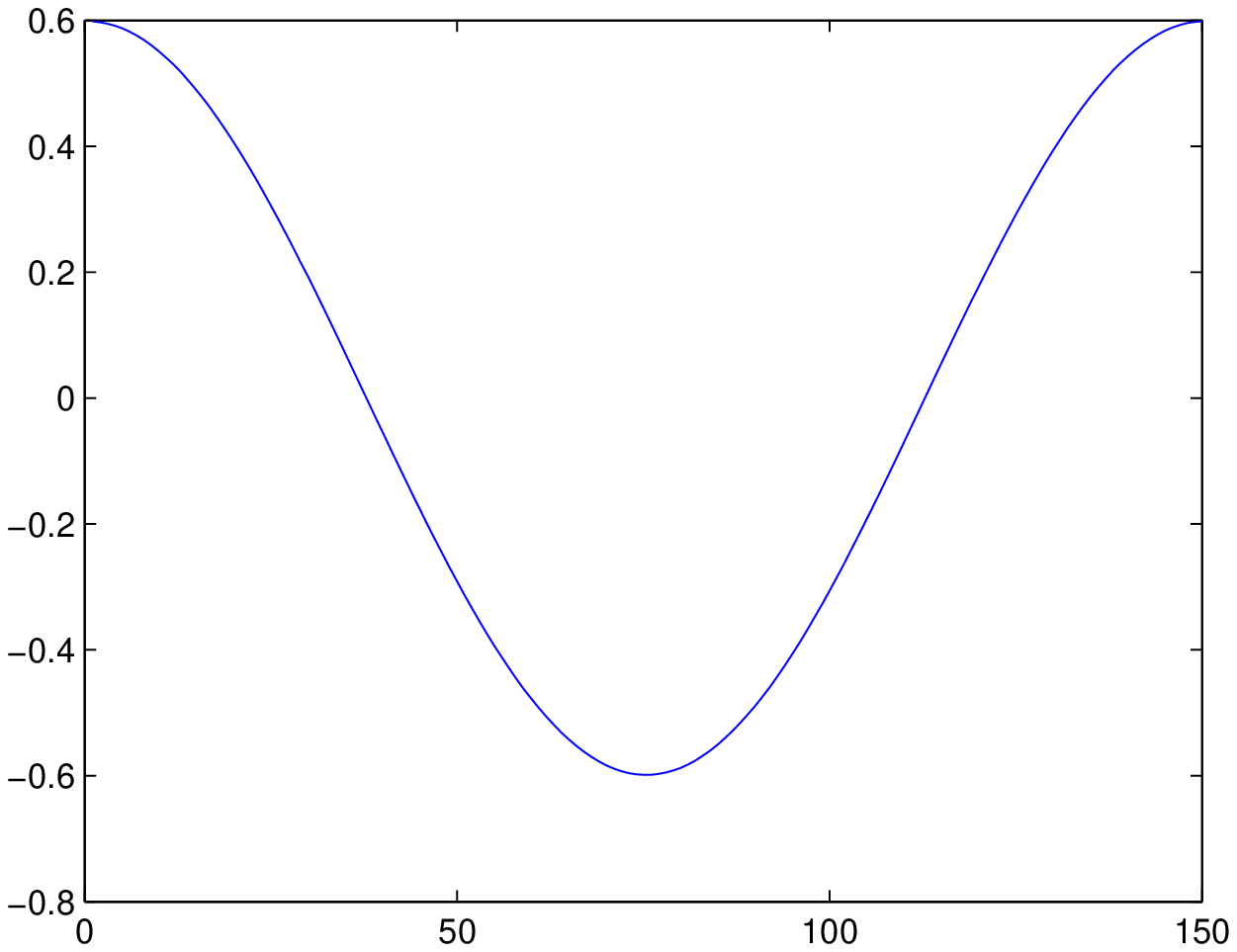}
    \vspace{-0.3cm}
\caption{\small{First example, from the left to the right, fields of velocity $v_1(x),\;v_5(x),\; v_{10}(x)$ for the lines $n=1,\;n=5$ and $n=10$.}} \label{figure4.3}
\end{center}
\end{figure}
\begin{figure}[!htb]
\begin{center}
	\includegraphics[height=4.0cm]{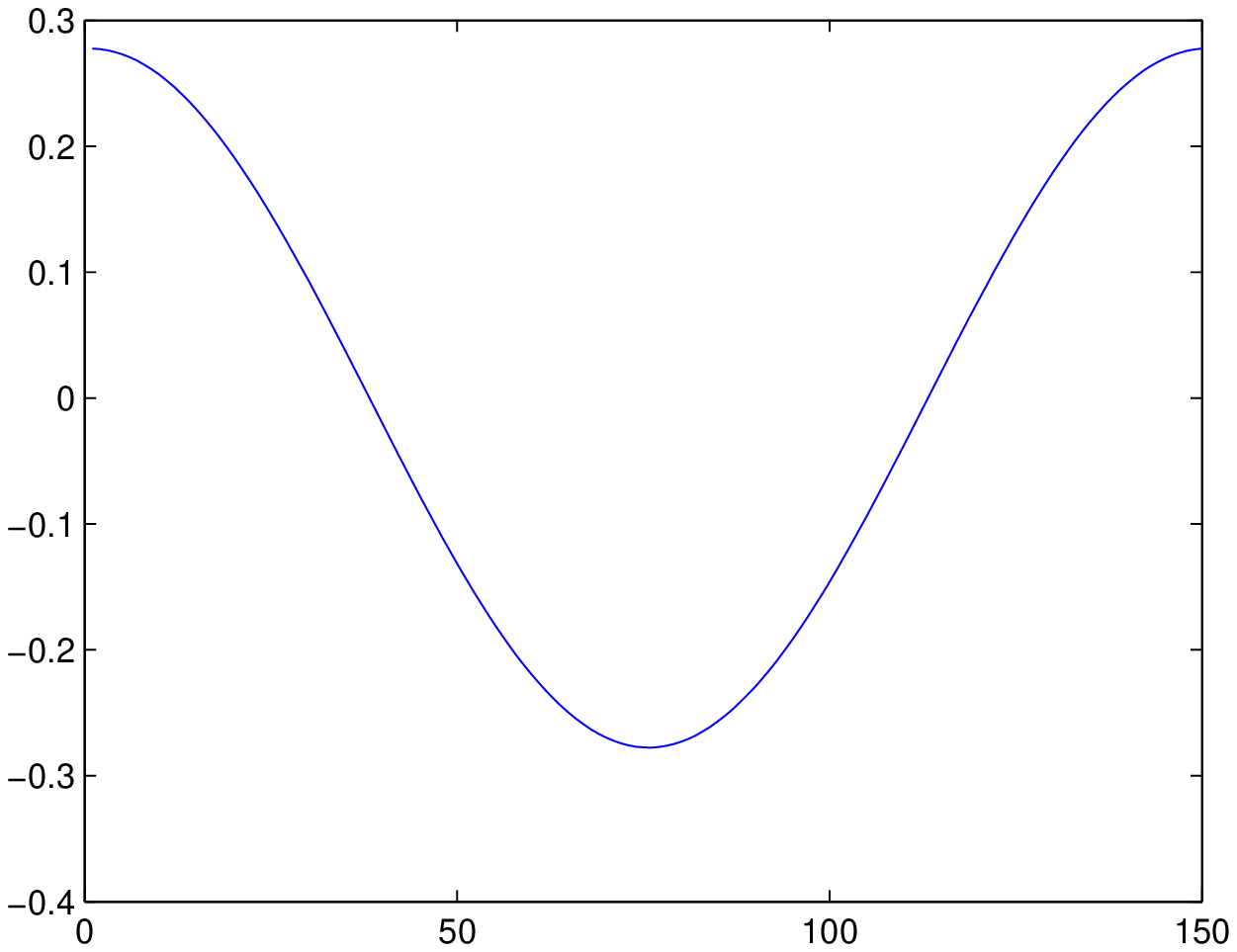}
    \includegraphics[height=4.0cm]{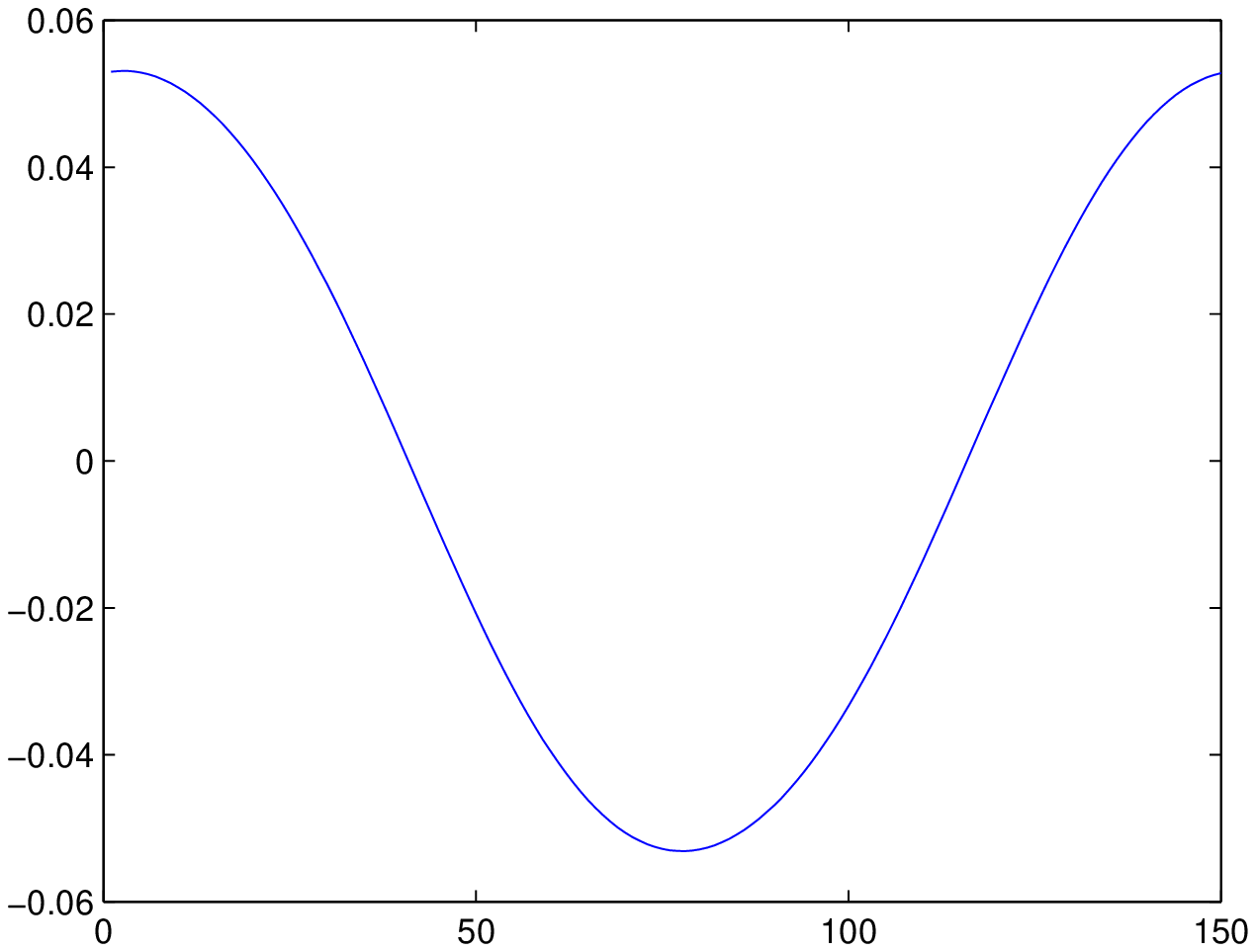}
    \vspace{-0.3cm}
\caption{\small{First example, from the left to the right, fields of velocity $v_{15}(x),\;v_{19}(x)$ for the lines $n=15$, and $n=19$.}} \label{figure4.4}
\end{center}
\end{figure}
%For the field of pressure $P$, we have got the lines:
\begin{figure}[!htb]
\begin{center}
	\includegraphics[height=4.0cm]{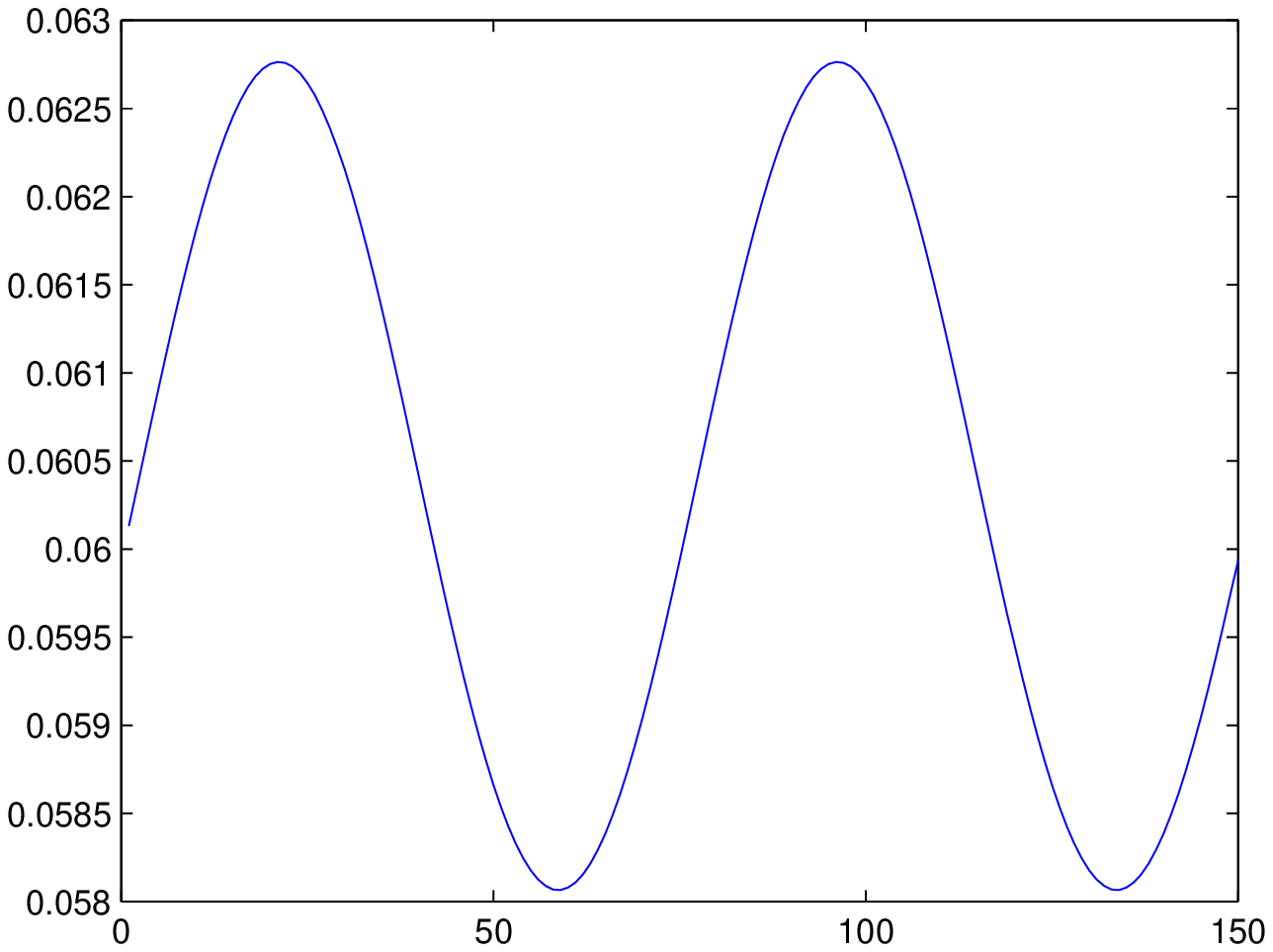}
    \includegraphics[height=4.0cm]{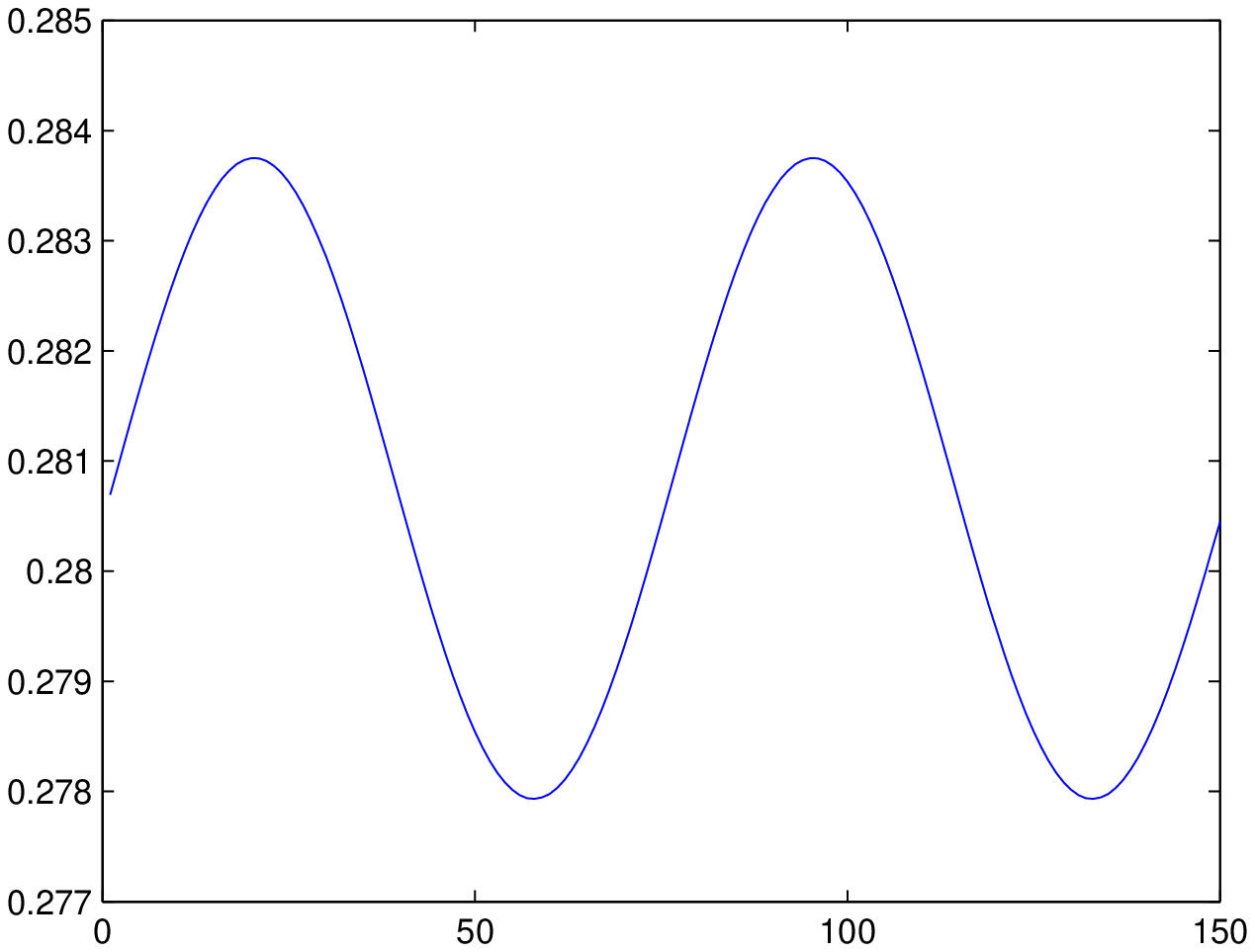}
    \includegraphics[height=4.0cm]{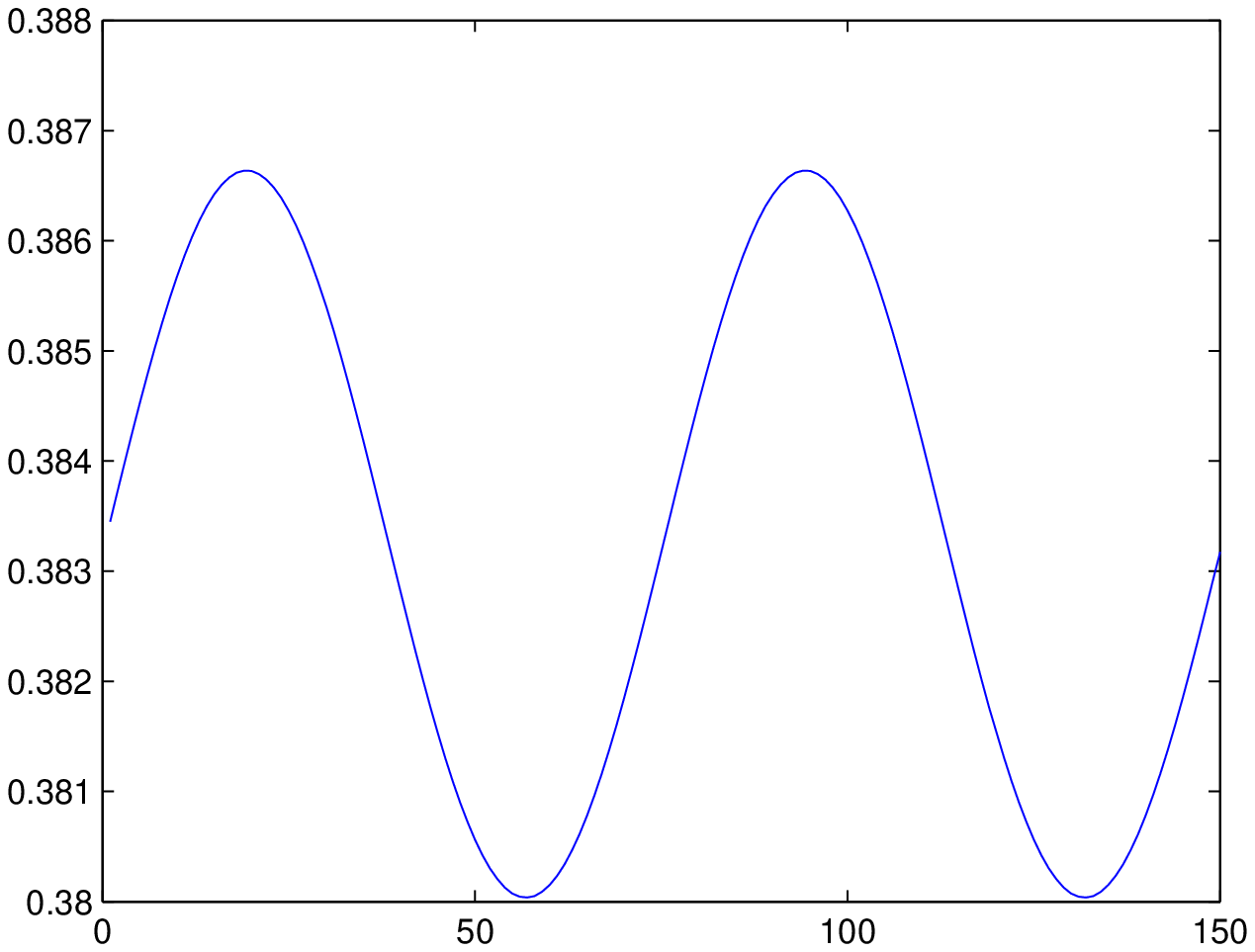}
    \vspace{-0.3cm}
\caption{\small{First example, from the left to the right, fields of pressure $P_1(x),\;P_5(x),\; P_{10}(x)$ for the lines $n=1,\;n=5$ and $n=10$.}} \label{figure4.5}
\end{center}
\end{figure}
\begin{figure}[!htb]
\begin{center}
	\includegraphics[height=4.0cm]{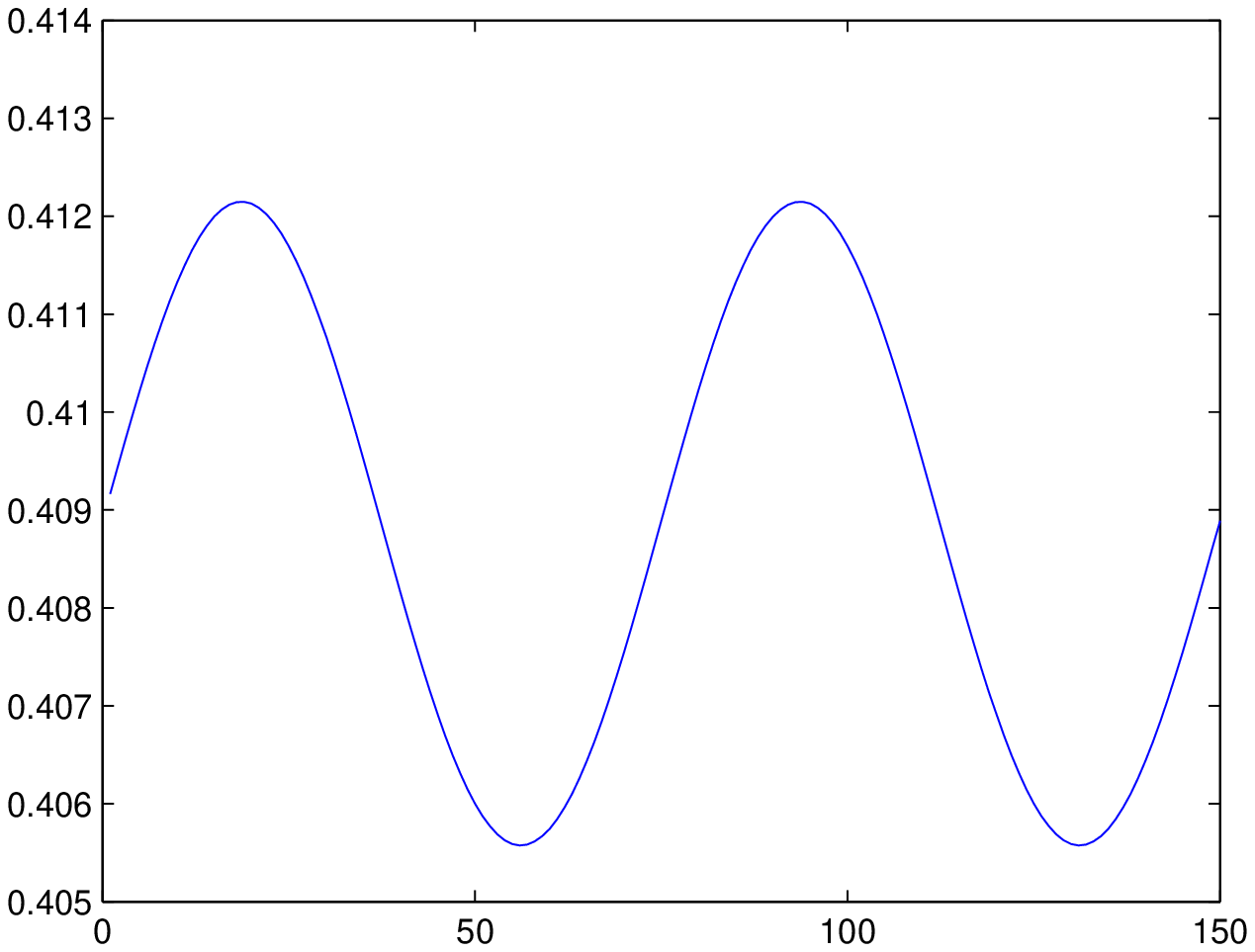}
    \includegraphics[height=4.0cm]{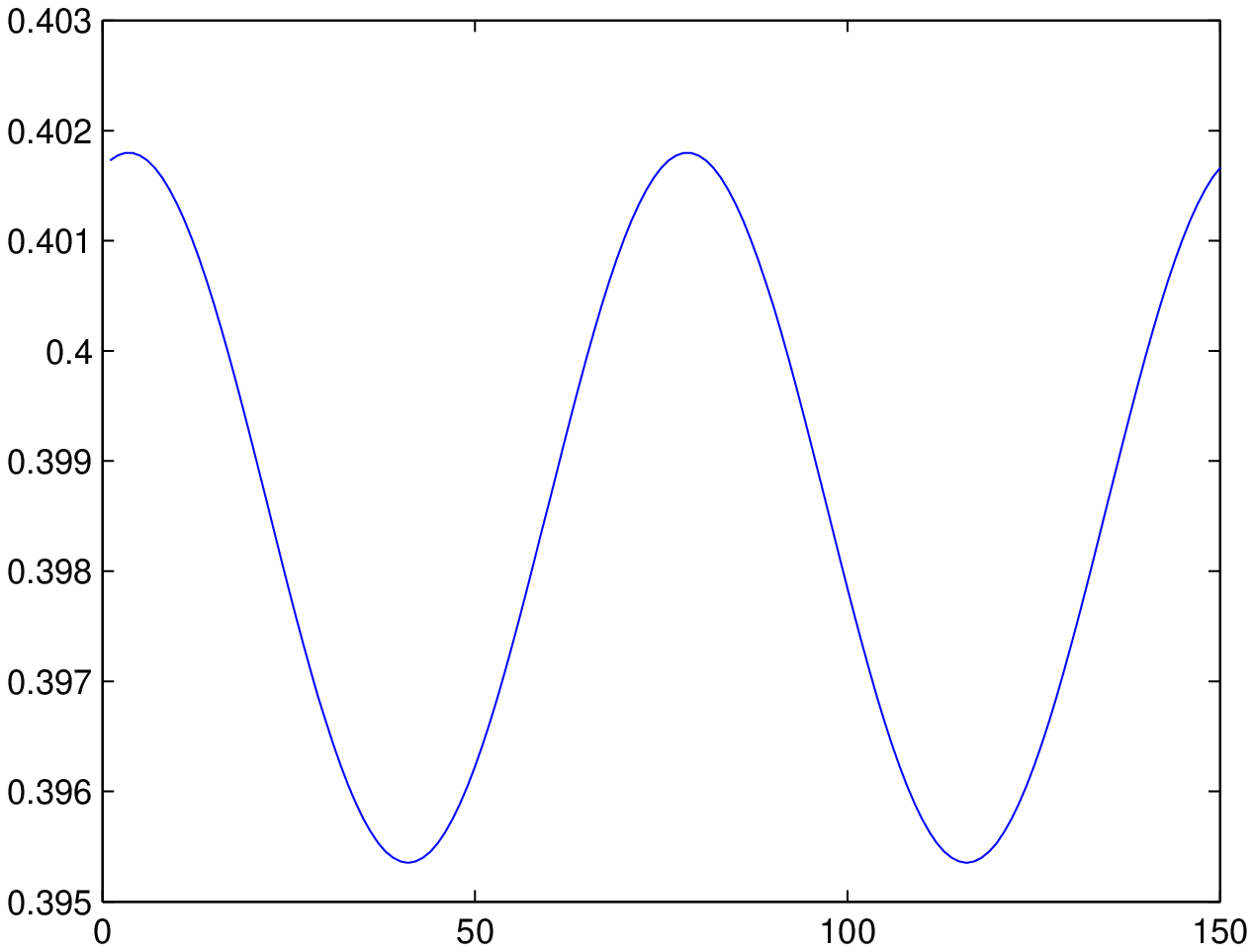}
    \vspace{-0.3cm}
\caption{\small{First example, from the left to the right, fields of pressure $P_{15}(x),\;P_{19}(x)$ for the lines $n=15$, and $n=19$.}} \label{figure4.6}
\end{center}
\end{figure}

For the second example, we consider $u_0(x)=-3.0\;\cos(x)\sin(x)$, $v_0(x)=-(-2.0 \cos(x)^2+\sin(x)^2)$ and $\nu=1.0$.% see figures from \ref{fig. 5.1} to  \ref{fig. 5.15}.

%Again denoting \begin{eqnarray}J&=&\int_\Omega \left(\nu \nabla^2 u-u \partial_x u-v \partial_y u- \partial_x P\right)^2\;d\Omega \nonumber \\
%&&+\int_\Omega \left(\nu \nabla^2 v-u \partial_x v-v \partial_y v- \partial_y P\right)^2\;d\Omega \nonumber \\
%&&+\int_\Omega \left(\partial_x u+ \partial_y v\right)^2\;d\Omega,
%\end{eqnarray}

Again the coefficients $\{a_i[n]\}, \{b_i[n]\}, \{c_i[n]\}$ have been obtained through the numerical minimization of $J(\{u_n\},\{v_n\},\{P_n\})$, so that
 for the mesh in question, we have obtained
$$ \text{ For this second example: } J(\{u_n\},\{v_n\},\{P_n\}) \approx 6.0* 10^{-7} \text{ for } \nu=1.0.$$
In any case, considering the values obtained for $J$, it seems we have got good  first approximations for the concerning solutions.

For the second example, for the field of velocities $u$ and $v$, and pressure field $P$, for the lines $n=1$, $n=5$, $n=10$, $n=15$ and $n=19$, please see figures \ref{figure4A.1} to \ref{figure4A.6}.
Once more, for all graphs, please consider units in $x$ to be multiplied by $2 \pi/150.$
\begin{figure}[!htb]
\begin{center}
	\includegraphics[height=4.0cm]{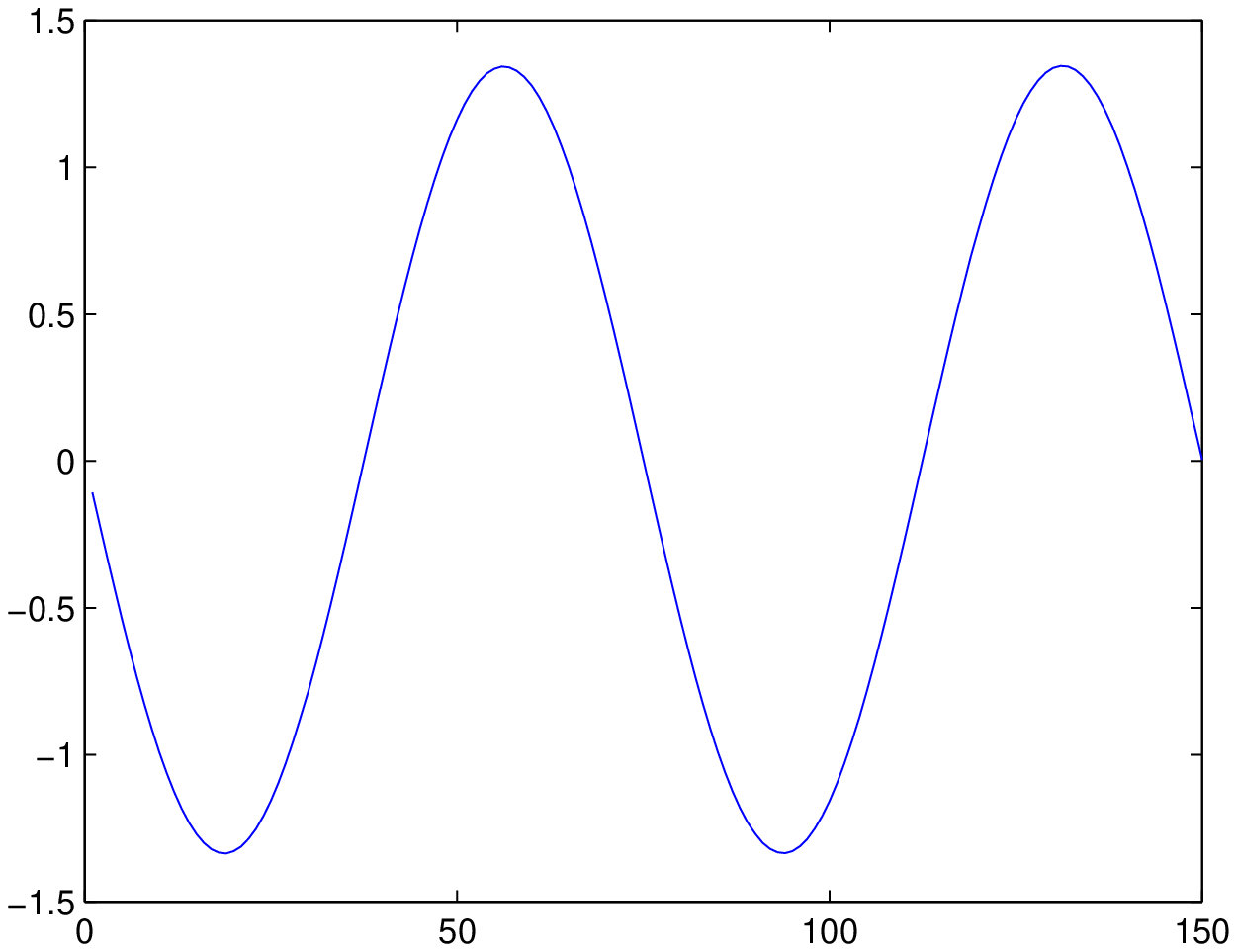}
    \includegraphics[height=4.0cm]{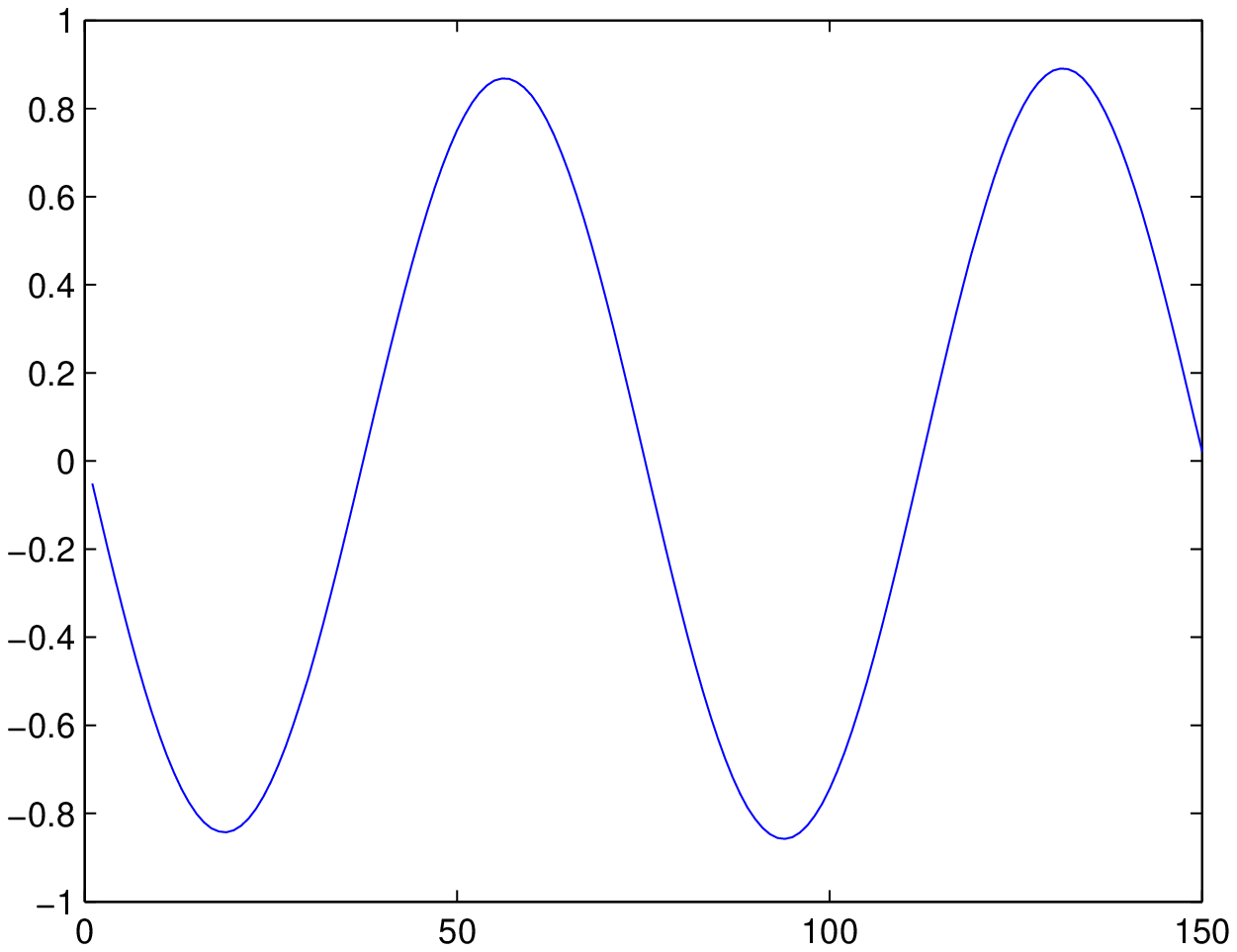}
    \includegraphics[height=4.0cm]{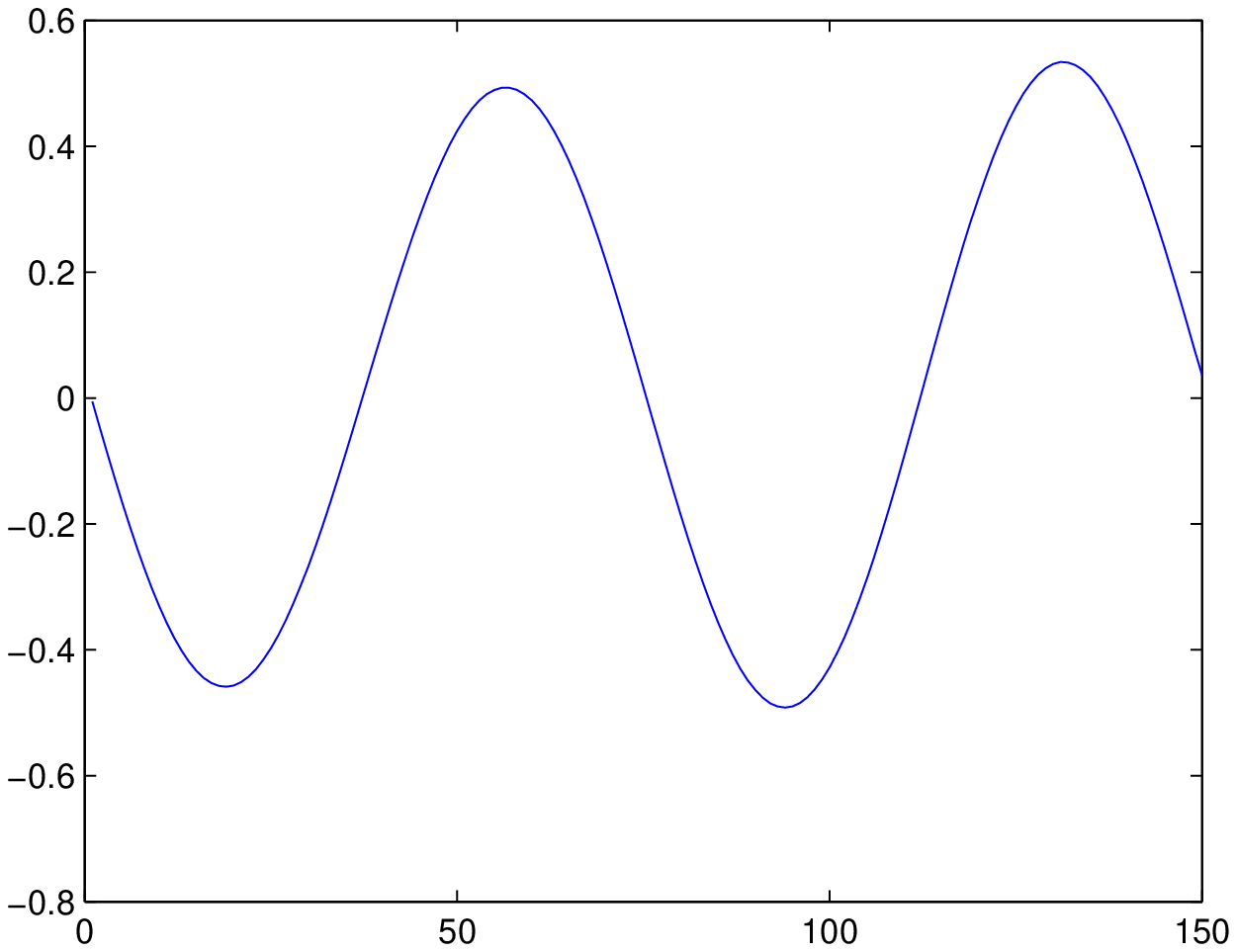}
    \vspace{-0.3cm}
\caption{\small{Second example, from the left to the right, fields of velocity $u_1(x),\;u_5(x),\; u_{10}(x)$ for the lines $n=1,\;n=5$ and $n=10$.}} \label{figure4A.1}
\end{center}
\end{figure}
\begin{figure}[!htb]
\begin{center}
	\includegraphics[height=4.0cm]{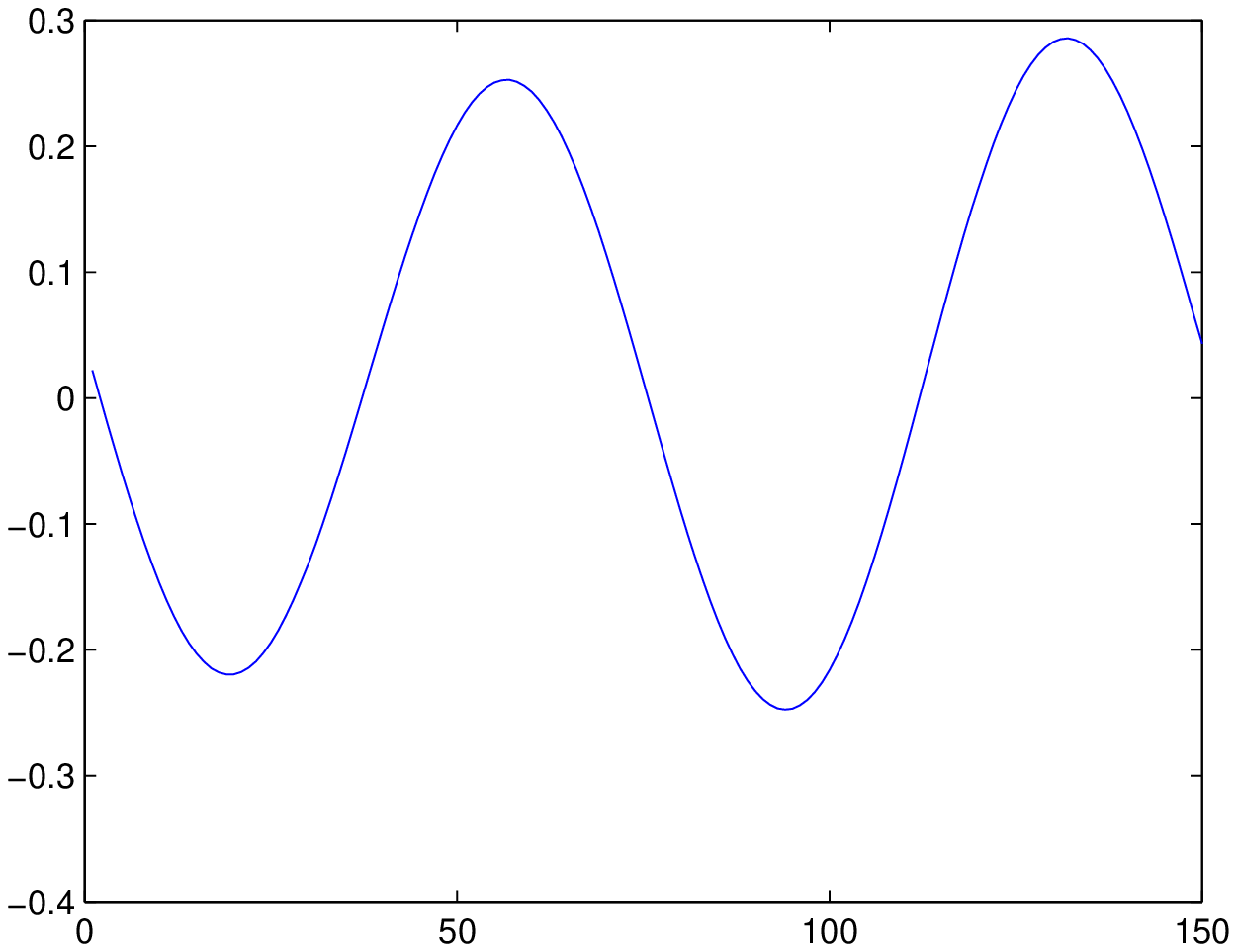}
    \includegraphics[height=4.0cm]{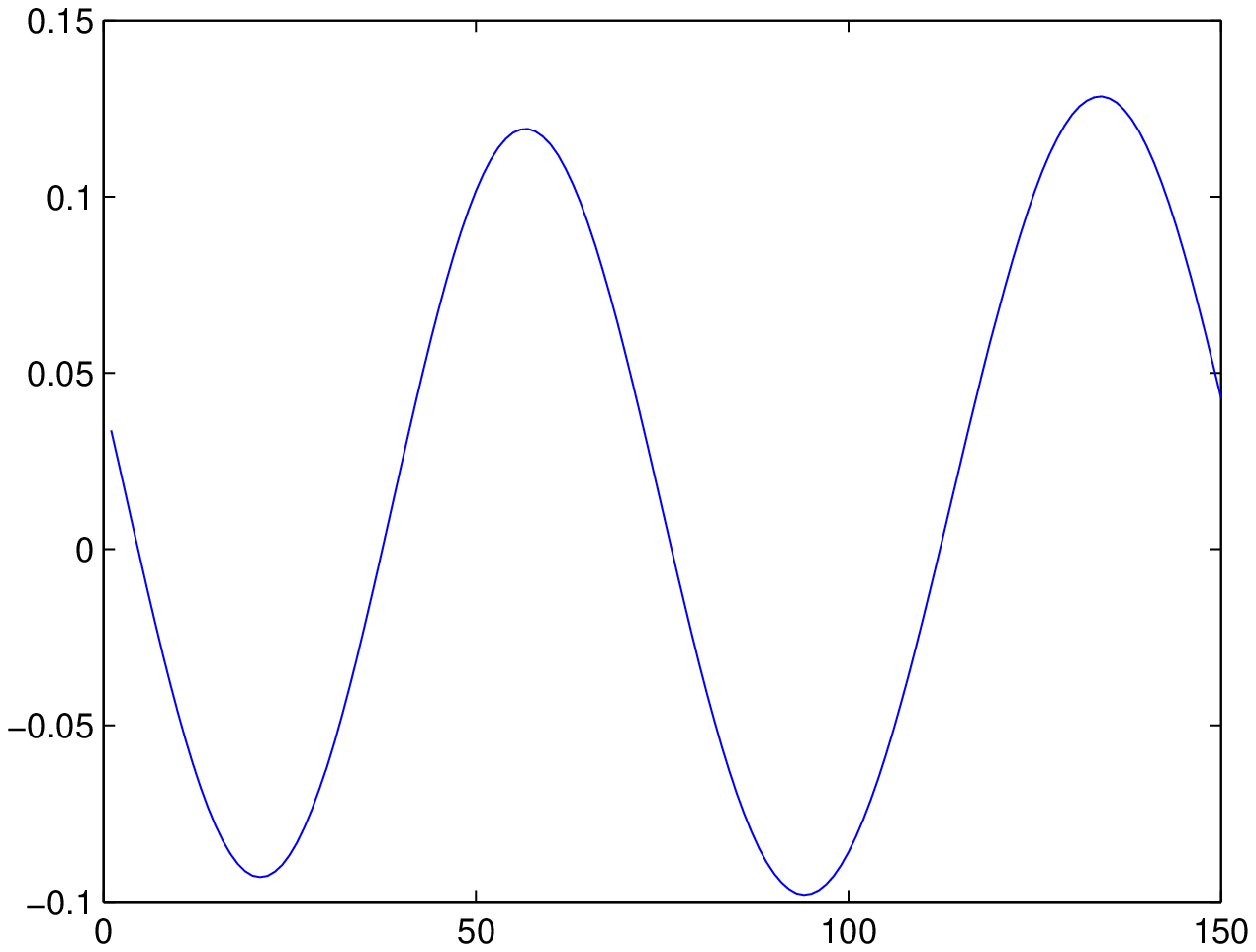}
    \vspace{-0.3cm}
\caption{\small{Second example, from the left to the right, fields of velocity $u_{15}(x),\;u_{19}(x)$ for the lines $n=15$, and $n=19$.}} \label{figure4A.2}
\end{center}
\end{figure}
\begin{figure}[!htb]
\begin{center}
	\includegraphics[height=4.0cm]{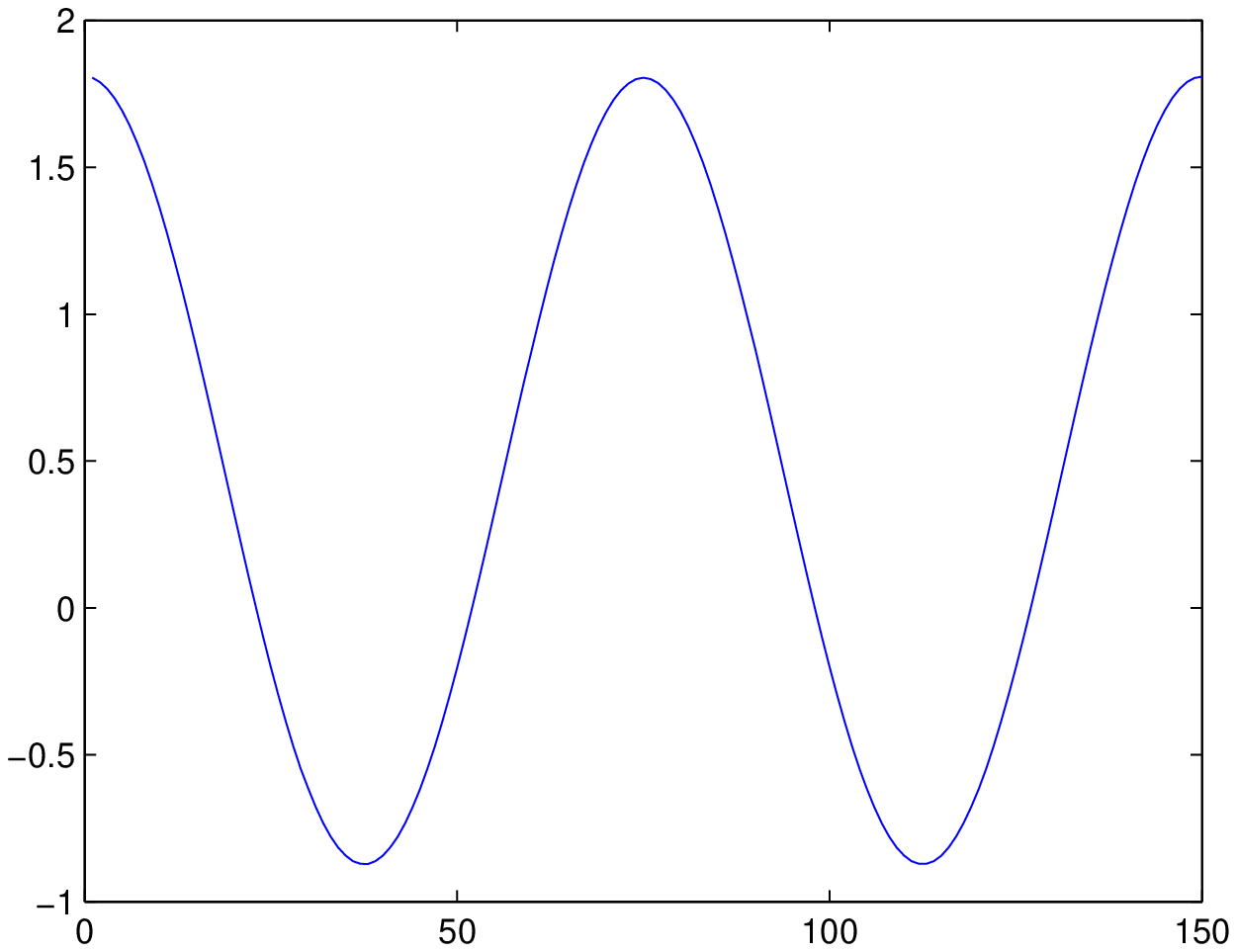}
    \includegraphics[height=4.0cm]{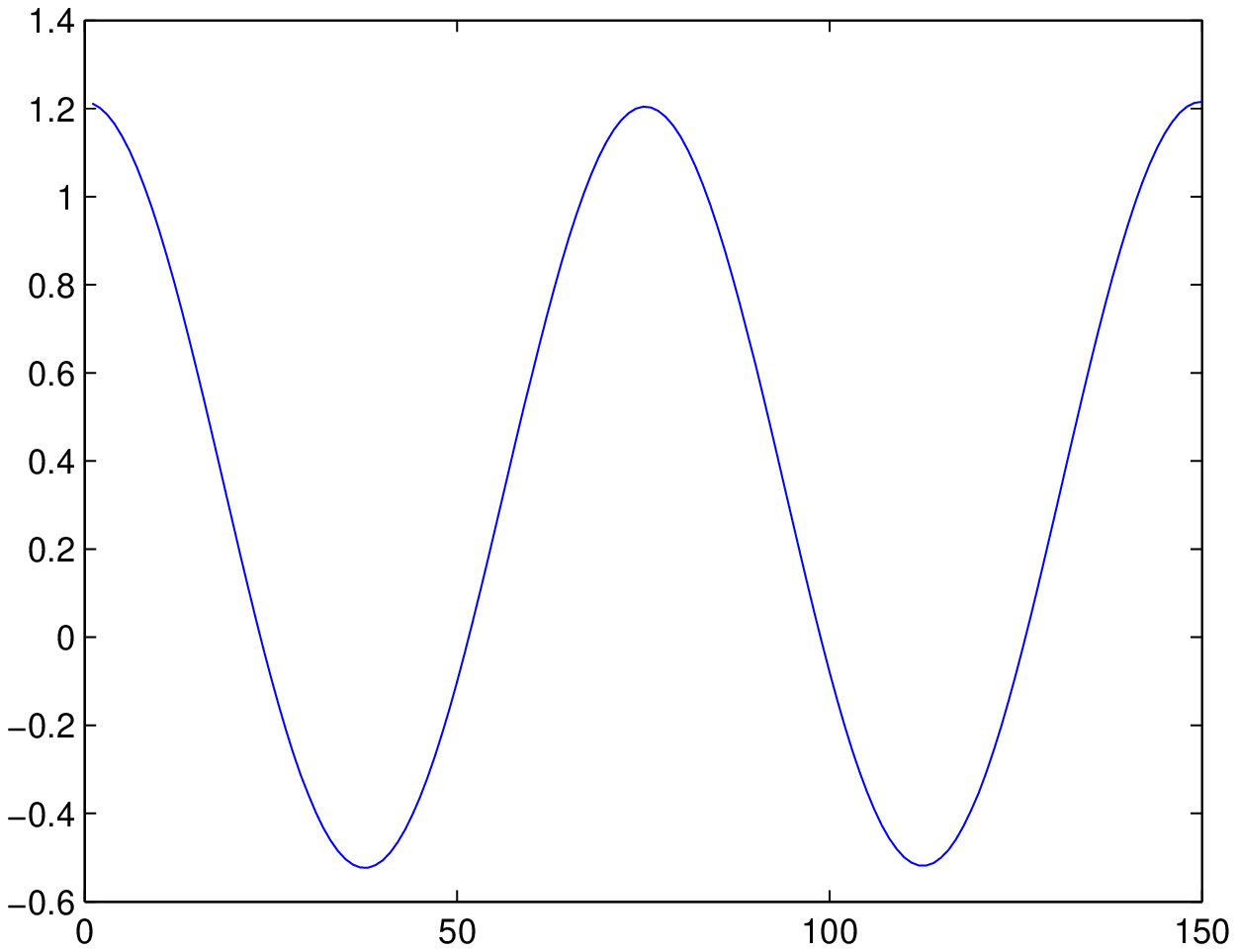}
    \includegraphics[height=4.0cm]{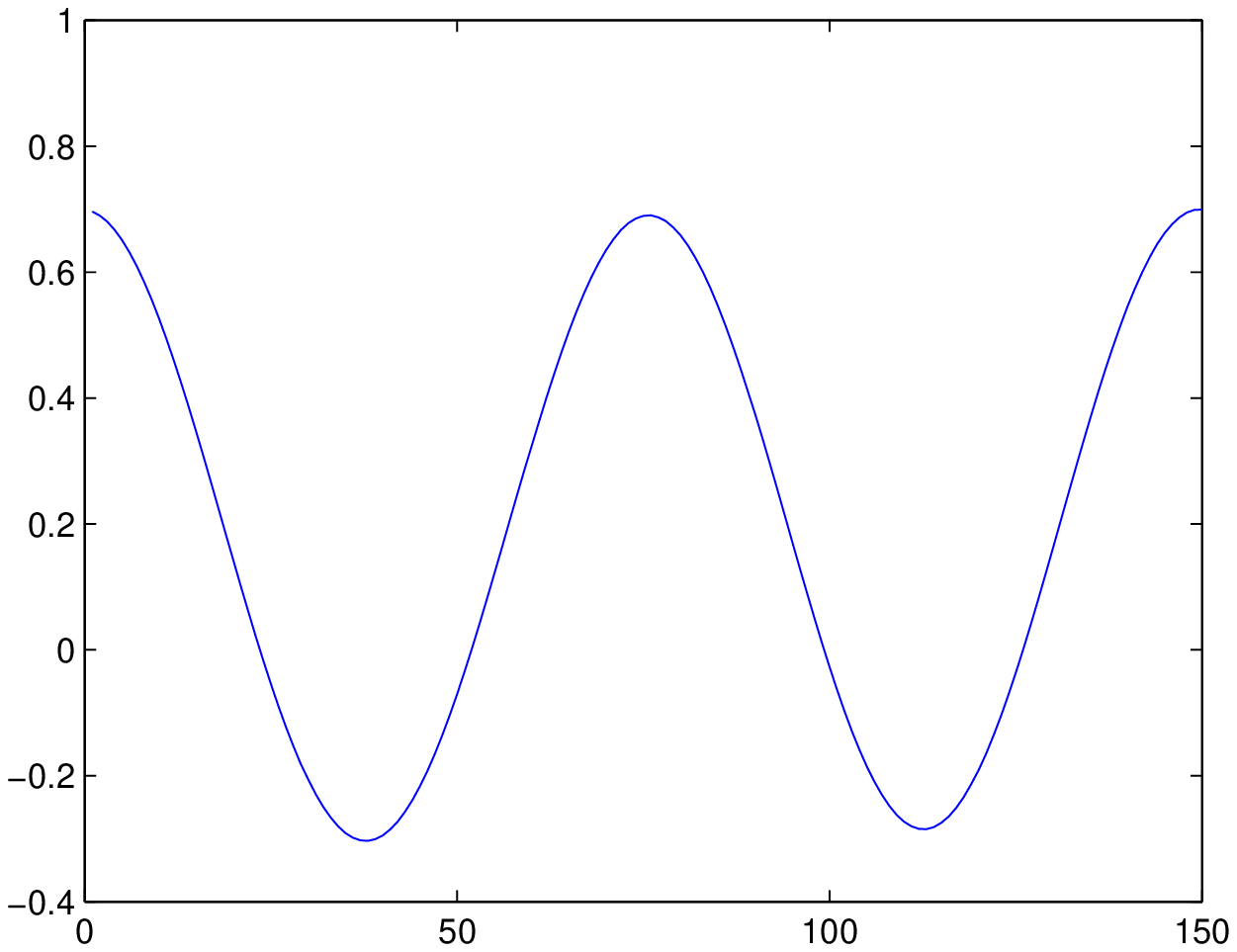}
    \vspace{-0.3cm}
\caption{\small{Second example, from the left to the right, fields of velocity $v_1(x),\;v_5(x),\; v_{10}(x)$ for the lines $n=1,\;n=5$ and $n=10$.}} \label{figure4A.3}
\end{center}
\end{figure}
\begin{figure}[!htb]
\begin{center}
	\includegraphics[height=4.0cm]{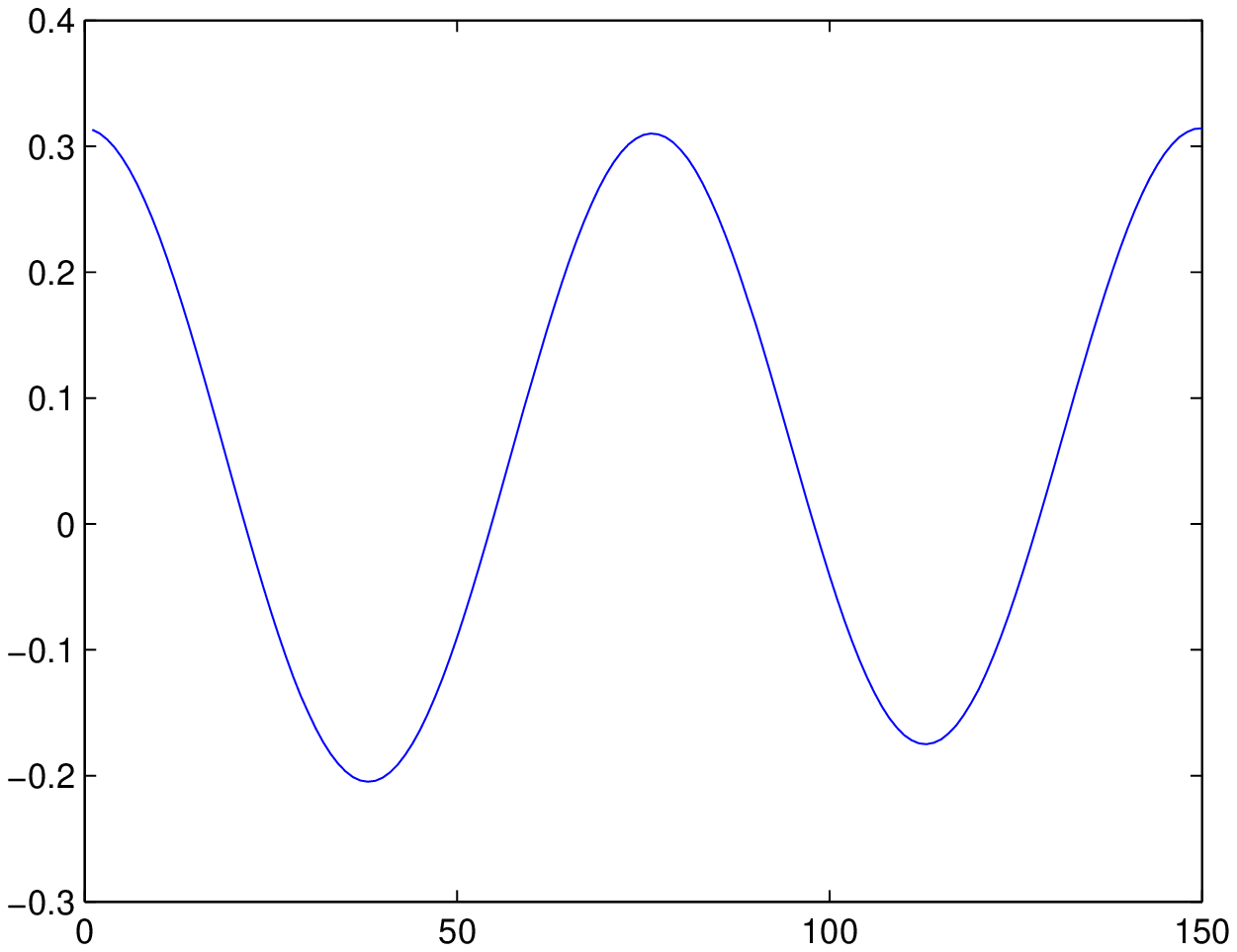}
    \includegraphics[height=4.0cm]{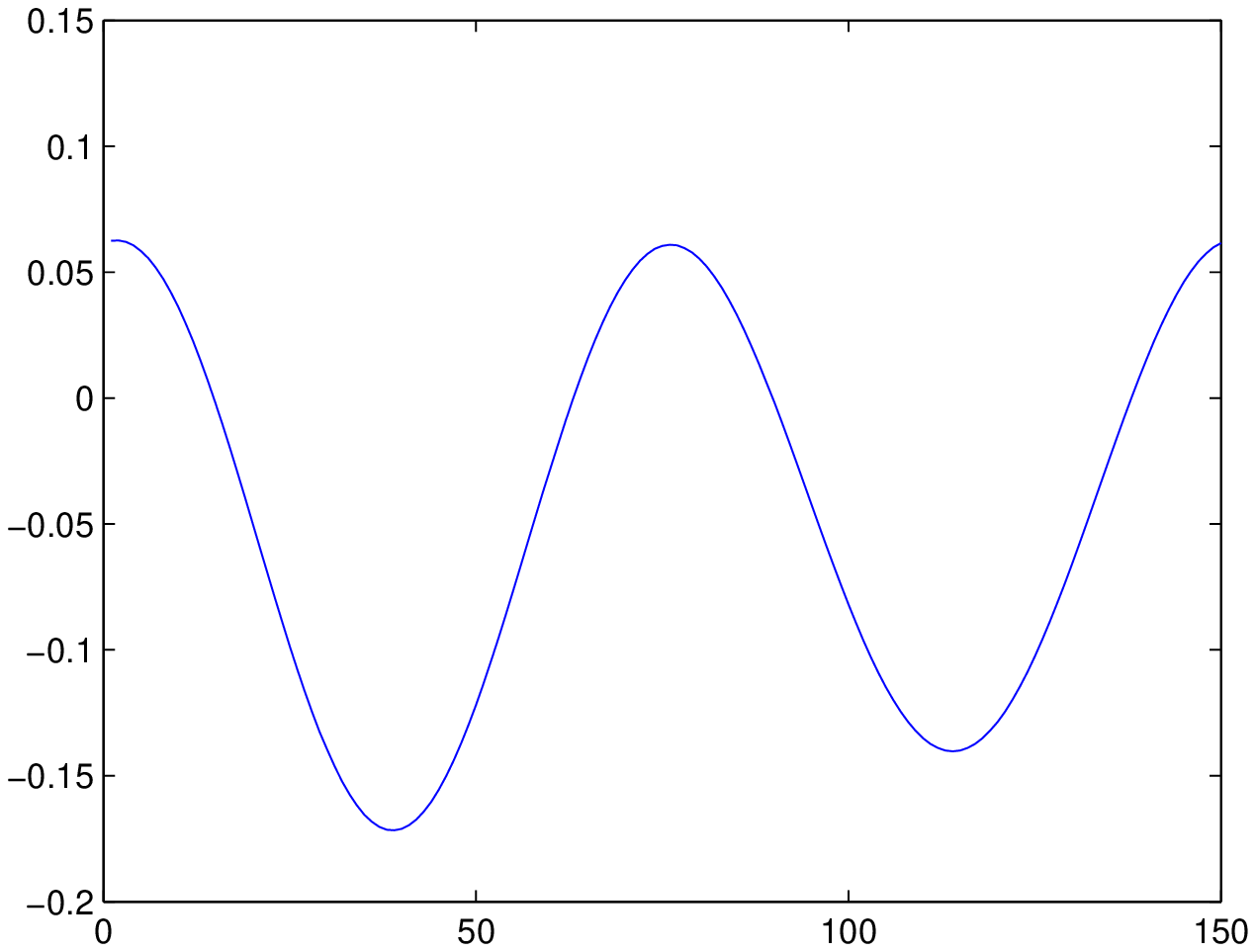}
    \vspace{-0.3cm}
\caption{\small{Second example, from the left to the right, fields of velocity $v_{15}(x),\;v_{19}(x)$ for the lines $n=15$, and $n=19$.}} \label{figure4A.4}
\end{center}
\end{figure}
%For the field of pressure $P$, we have got the lines:
\begin{figure}[!htb]
\begin{center}
	\includegraphics[height=4.0cm]{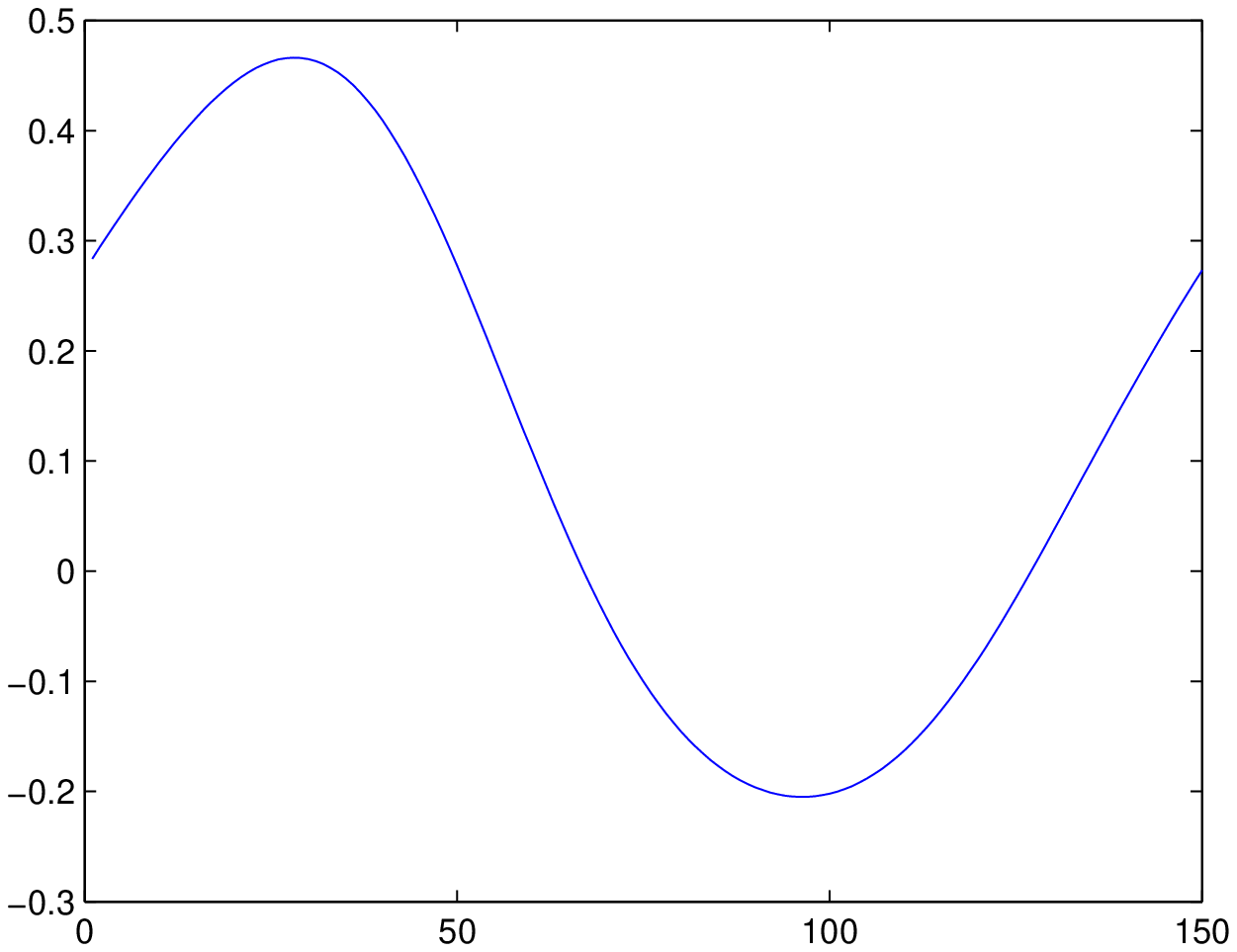}
    \includegraphics[height=4.0cm]{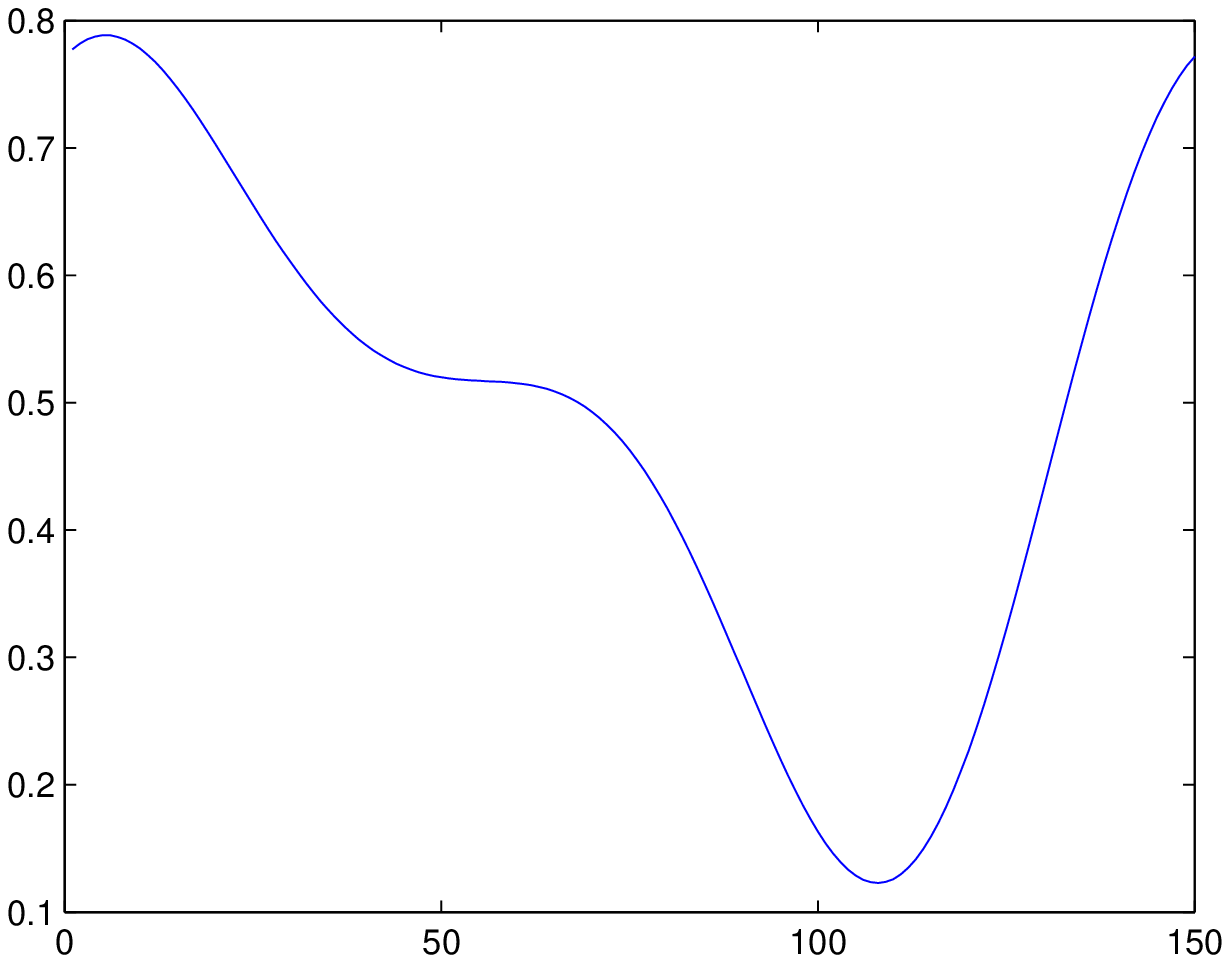}
    \includegraphics[height=4.0cm]{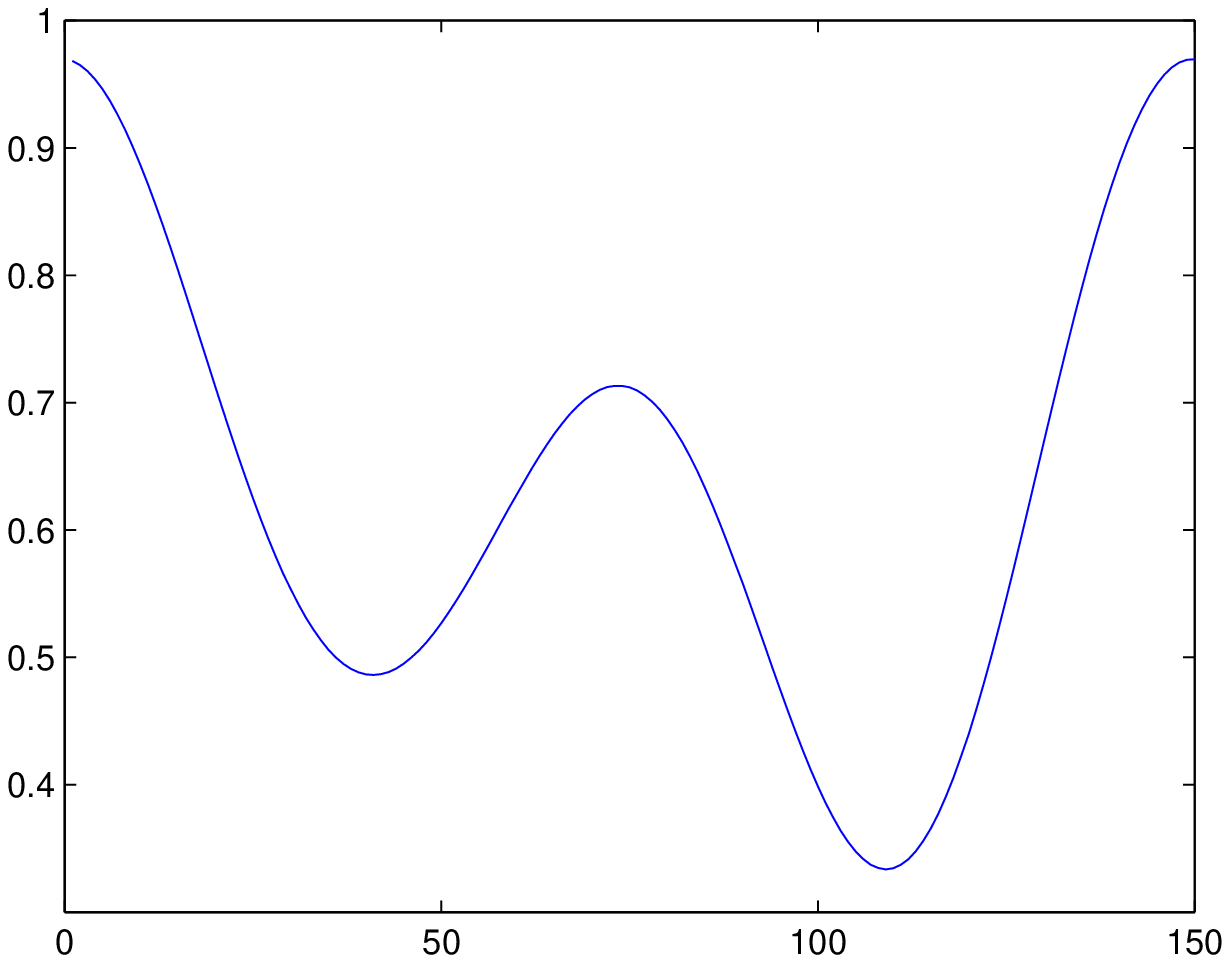}
    \vspace{-0.3cm}
\caption{\small{Second example, from the left to the right, fields of pressure $P_1(x),\;P_5(x),\; P_{10}(x)$ for the lines $n=1,\;n=5$ and $n=10$.}} \label{figure4A.5}
\end{center}
\end{figure}
\begin{figure}[!htb]
\begin{center}
	\includegraphics[height=4.0cm]{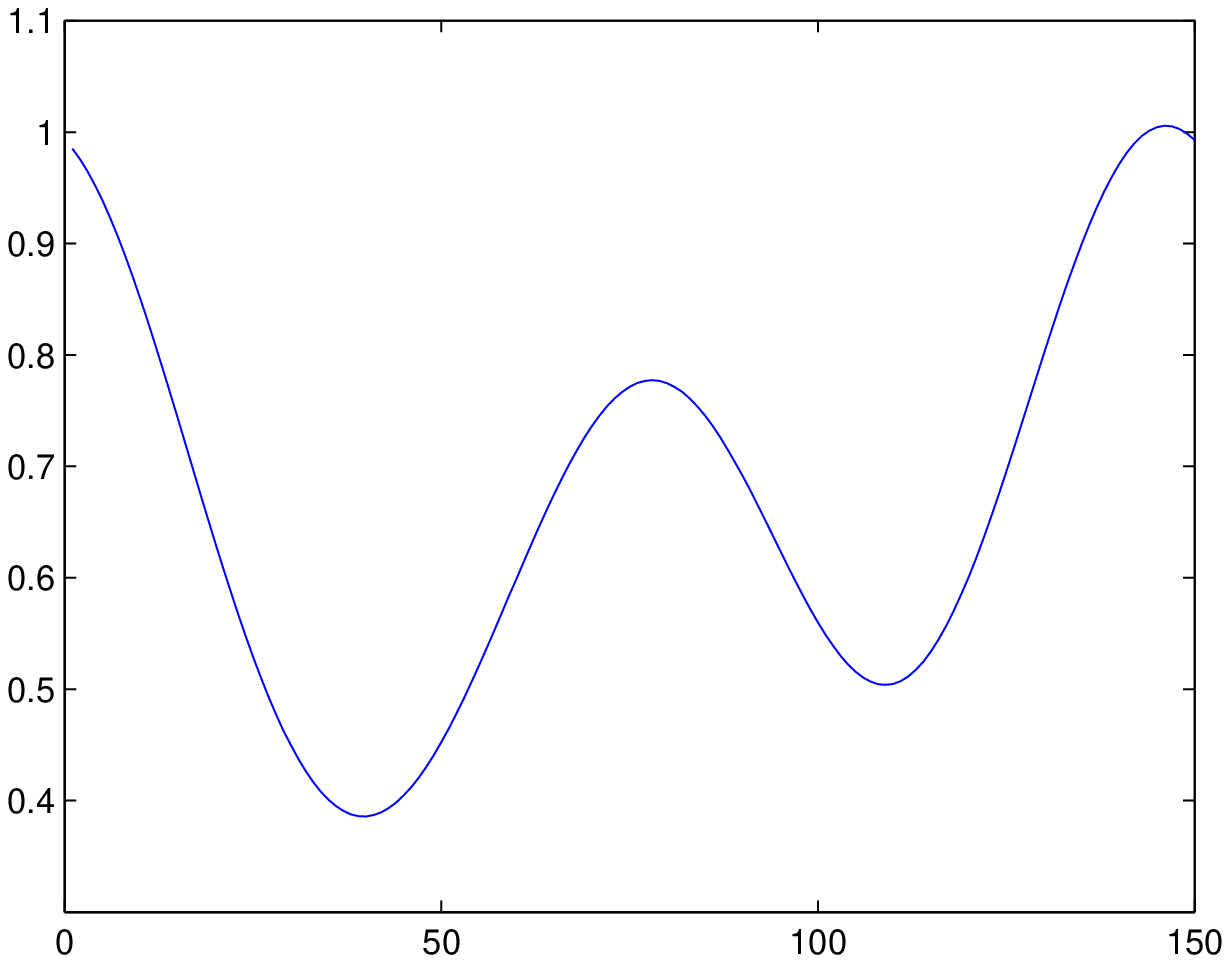}
    \includegraphics[height=4.0cm]{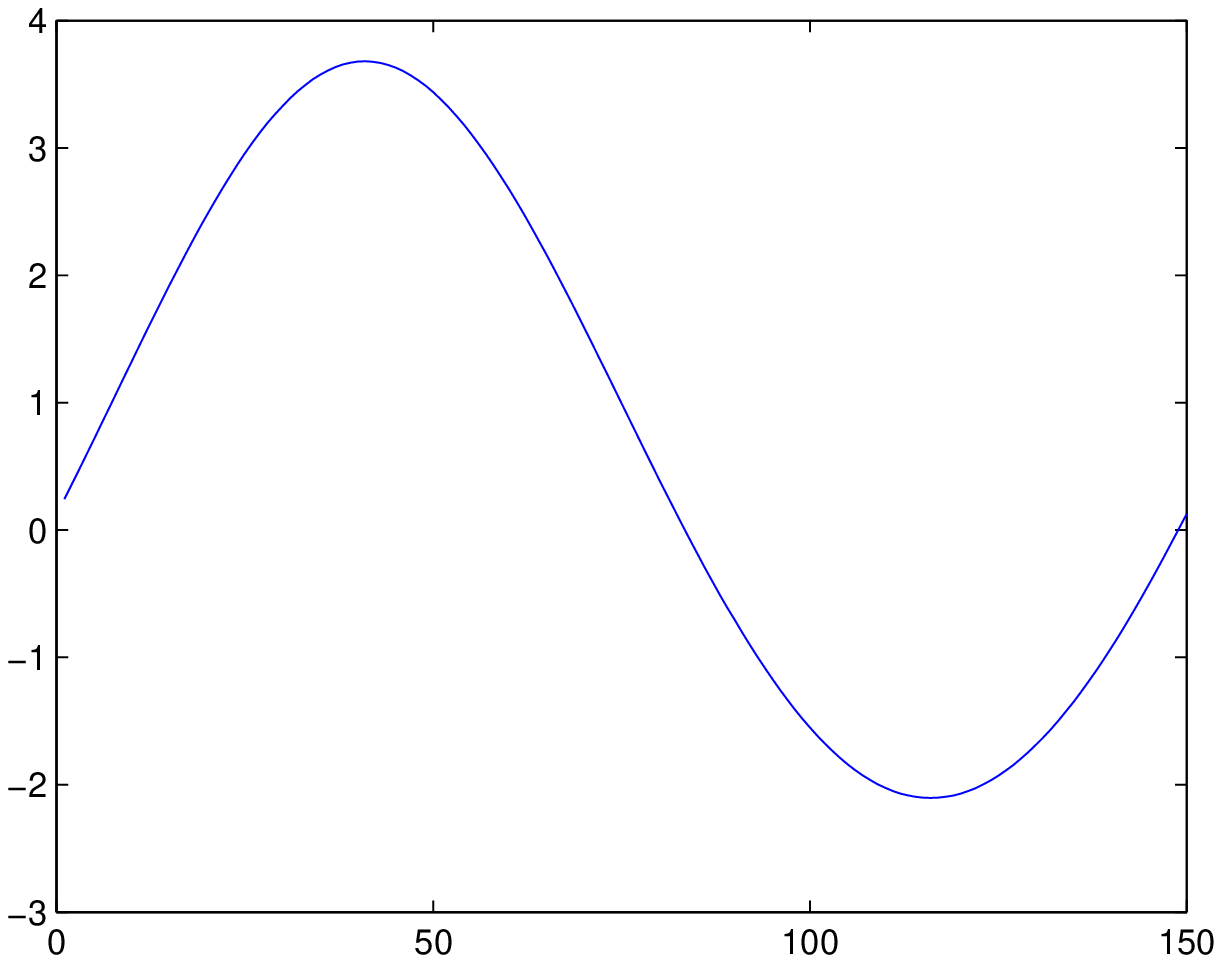}
    \vspace{-0.3cm}
\caption{\small{Second example, from the left to the right, fields of pressure $P_{15}(x),\;P_{19}(x)$ for the lines $n=15$, and $n=19$.}} \label{figure4A.6}
\end{center}
\end{figure}

\section{Conclusion}
In the first part of this article, we obtain a linear system whose the solution solves the time-independent incompressible
 Navier-Stokes system. In the second part, we develop solutions for  two-dimensional examples also for the time-independent incompressible Navier-Stokes system,
 through the generalized method of lines. Considering the values for $J$ obtained, we have got very good approximate solutions
 for the model in question, in a finite differences context.
  The extension of such results to $\mathbb{R}^3$, compressible
and time dependent cases is planned for a future work.

%Non-BibTeX users please

\end{document}